\documentclass[paper=a4,pagesize,DIV=calc, 12pt,reqno,numbers=endperiod]{scrartcl}

\usepackage[T1]{fontenc}
\usepackage[utf8]{inputenc}

\usepackage[english]{babel}
\usepackage{fixltx2e}

\usepackage{%
  ellipsis,
  ragged2e,
 marginnote
}
\usepackage[tracking=true]{microtype}%
\DeclareMicrotypeSet*[tracking]{my}%
  { font = */*/*/sc/* }%
\SetTracking{ encoding = *, shape = sc }{ 45 }

\usepackage{lmodern}

\usepackage{stmaryrd}
\usepackage{mathrsfs}

\usepackage{amsmath}
\usepackage{amsfonts}
\usepackage{amssymb}

\usepackage[nobysame,initials]{amsrefs}

\usepackage{amsthm}
\theoremstyle{plain}
\newtheorem{theorem}{Theorem}[section]
\newtheorem*{theorem*}{Theorem}
\newtheorem{lemma}[theorem]{Lemma}
\newtheorem{proposition}[theorem]{Proposition}
\newtheorem{corollary}[theorem]{Corollary}

\theoremstyle{definition}
\newtheorem{definition}[theorem]{Definition}
\newtheorem{example}[theorem]{Example}
\newtheorem*{example*}{Example}

\theoremstyle{remark}

\numberwithin{equation}{section}

\usepackage{xcolor}
\usepackage{graphicx}
\usepackage{tikz}
\usetikzlibrary{arrows,positioning}

\usepackage{mleftright}\mleftright
\usepackage[
	bookmarks=true,
	bookmarksnumbered=true,
	unicode=false,
	pdftoolbar=true,
	pdfmenubar=true,
	pdffitwindow=false,
	pdfstartview={FitH},
	pdftitle={The floating body in real space forms},
	pdfauthor={Florian Besau},
	pdfsubject={2010 MSC classes:  52A55 (Primary), 28A75, 52A20, 53A35 (Secondary)},
	pdfcreator={},
	pdfproducer={},
	pdfkeywords={hyperbolic convex geometry} {hyperbolic convex floating body} {affine surface area} {floating measure},
	pdfnewwindow=true,
	colorlinks=true,
	linkcolor=black,
	citecolor=black,
	filecolor=blue,
	urlcolor=black]{hyperref}

\newcommand{\vol}{\mathrm{vol}}

\newcommand{\bd}{\mathrm{bd}}

\newcommand{\F}{\mathcal{F}}

\newcommand{\Sp}{\mathrm{Sp}}
\newcommand{\R}{\mathbb{R}}
\renewcommand{\S}{\mathbb{S}}
\renewcommand{\H}{\mathbb{H}}
\newcommand{\B}{\mathbb{B}}

\newcommand{\inter}{\mathrm{int}}
\renewcommand{\exp}{\mathrm{exp}}

\newcommand{\M}{\mathfrak{M}}
\newcommand{\K}{\mathcal{K}}

\begin{document}

\author{\large Florian Besau\footnote{ supported by the European Research Council (ERC) 
within the
project ``Isoperimetric Inequalities and Integral Geometry''
Project number: 306445.}
\and 
\large Elisabeth M. Werner\footnote{ partially supported by an NSF grant.}
}
\date{}

\title{\Large The Floating Body in Real Space Forms}
\maketitle

\let\thefootnote\relax\footnotetext{
	2010 MSC classes:  52A55 (Primary), 28A75, 52A20, 53A35 (Secondary)
}

\vspace{-1cm}
\begin{abstract}
\addcontentsline{toc}{section}{Abstract}
\noindent\footnotesize \textbf{Abstract.} \noindent
We carry out a systematic investigation on floating bodies in real space forms.
A new unifying approach not only allows us to treat the important classical case 
of Euclidean space as well as the recent extension to the Euclidean unit sphere,
but also the new extension of floating bodies to hyperbolic space.

Our main result establishes a relation between the derivative of the volume of the
floating body and a certain surface area measure, which we called the floating area.
In the Euclidean setting the floating area coincides with the well known affine surface area,
a powerful tool in the affine geometry of convex bodies.
\end{abstract}

\renewcommand{\citeform}[1]{\textbf{#1}}

\section{Introduction}

Two important closely related notions in affine convex geometry are the floating
body and the affine surface area of a convex body.
The floating body of a convex body is obtained by cutting off caps 
of volume less or equal to a fixed positive constant $\delta$.
Taking the right-derivative of the volume of the floating body gives rise to the affine surface area. 
This was established for all convex bodies in all dimensions by Schütt and Werner
in \cite{Schuett:1990}.

The affine surface area was introduced by Blaschke in 1923 \cite{Blaschke:1923}.
Due to its important properties, which make it an effective and powerful tool, it
is omnipresent in geometry.
The affine surface area and its generalizations in the rapidly developing $L_p$ and
Orlicz Brunn--Minkowski theory are the focus of intensive investigations (see e.g.\
\cites{Lutwak:1993,Lutwak:1996,Gardner:2015,Gardner:2014,Werner:2010, Werner:2011,
Chou:2006,Gardner:1998,Ye:2015,Ye:2015a,Schuett:2004}).

A first characterization of affine surface area was achieved by Ludwig and Reitzner
\cite{Ludwig:1999} and had a profound impact on valuation theory of convex bodies. 
They started a line of research 
(see e.g.\ \cites{Ludwig:2010,Ludwig:2010a,Ludwig:2010c,Parapatits:2012,Schuster:2012,
Haberl:2012,Parapatits:2013})
leading up to the very recent characterization of all centro-affine valuations by Haberl
and Parapatits \cite{Haberl:2014}.

There is a natural inequality associated with affine surface area, the affine isoperimetric
inequality, which states that among all convex bodies, with fixed volume, affine surface 
area is maximized for ellipsoids. 
This inequality has sparked interest into affine isoperimetric inequalities with a 
multitude of results 
(see e.g.\ \cites{Lutwak:1996,Lutwak:2000, Lutwak:2002, Haberl:2009, Haberl:2009a, 
Haberl:2012,Zhang:1999,Cianchi:2009, Bernig:2014, Ye:2015, Ye:2015a, Livshyts:2015,
Werner:2010}).

There are numerous other applications for affine surface area, such as, 
the approximation theory of convex bodies by polytopes 
\cites{Gruber:1988,Gruber:1993, Ludwig:1999a, Boeroeczky:2000, Boeroeczky:2000a,
Schuett:1991, Schuett:2000, Schuett:2003, Schuett:2004, Paouris:2012, Werner:2012, 
Paouris:2013, Reitzner:2002}, 
affine curvature flows \cites{Andrews:1996,Andrews:1999,Ivaki:2013,Ivaki:2013a,Ivaki:2015},
information theory \cites{Caglar:2014, Caglar:2014a, Caglar:2015, Artstein-Avidan:2012} and 
partial differential equations \cite{Lutwak:1995}.

In this paper we introduce the floating bodies for spaces of constant curvature,
i.e., real space forms.
Our considerations lead to a new surface area measure for convex bodies, which
we call the floating area.
This floating area is intrinsic to the constant curvature space and not only
coincides with affine surface area in the flat case, but also has similar properties
in the general case.
Namely, the floating area is a valuation and upper semi-continuous.

We lay the foundation for further investigations of floating bodies and the
floating area of convex bodies in more general spaces. The authors believe that
both notions are of interest in its own right and will in particular be useful
for applications, such as, isoperimetric inequalities and approximation theory
of convex bodies in spaces of constant curvature.

\subsection{Statement of principal results}

A real space form is a simply connected complete Riemannian manifold with constant sectional curvature 
$\lambda$. For $\lambda\in \R$ and $n\in\mathbb{N}, n\geq2$, we denote by $\Sp^n(\lambda)$ the real space form of 
dimension $n$ and curvature $\lambda$. 
This includes the special cases of the sphere $\S^n=\Sp^n(1)$, hyperbolic space $\H^n=\Sp^n(-1)$ and Euclidean 
space $\R^n=\Sp^n(0)$.
A compact (geodesically) convex set $K$ is called a convex body. 
The set of convex bodies in $\R^n$ with non-empty interior is denoted by $\mathcal{K}_0(\R^n)$, or 
$\mathcal{K}_0(A)$ if we consider convex bodies contained in an open subset $A\subset\R^n$. The set of convex 
bodies in a space form with non-empty interior is denoted by $\K_0(\Sp^n(\lambda))$, $\K_0(\S^n)$ or 
$\K_0(\H^n)$. For further details we refer to Section 3.

A hyperplane in a real space form $\Sp^n(\lambda)$ is a totally geodesic hypersurface. It is isometric to $\Sp^{n-1}(\lambda)$.
Hyperplanes split the space into two open and connected parts which are half-spaces. We 
denote by $H^+$ and $H^-$ the closed half-spaces bounded by the hyperplane $H$.
The standard volume measure on $\Sp^n(\lambda)$ is $\vol_n^\lambda$.

\begin{definition}[$\lambda$-Floating Body]
Let $\lambda\in\R$ and $K\in\K_0(\Sp^n(\lambda))$. For $\delta>0$ the \emph{$\lambda$-floating body} 
$\F_\delta^\lambda\, K$ is defined by
\begin{align*}
	\F_\delta^\lambda\, K =\bigcap \left\{ H^- : \vol_n^\lambda\left(K\cap H^+\right)\leq 
\delta^{\frac{n+1}{2}}\right\}.
\end{align*}
\end{definition}

The main theorem of this article is the following:
\begin{theorem}\label{thm:intro_main}
	Let $n\geq 2$. If $K\in\K_0(\Sp^n(\lambda))$, then
	the right-derivative of $\vol_n^\lambda(\F_\delta^\lambda\, K)$ at $\delta=0$ exists. More precisely, we 
have
	\begin{align*}
		\lim_{\delta\to 0^+} \frac{\vol_n^\lambda(K)-\vol_n^\lambda(\F_\delta^\lambda\, K)}{\delta} = c_n 
\Omega^\lambda(K),
	\end{align*}
	where $c_n=\tfrac{1}{2}\left((n+1)/\kappa_{n-1}\right)^{2/(n+1)}$ and
	\begin{align*}
		\Omega^\lambda(K) = \int\limits_{bd\, K} H_{n-1}^{\lambda}(K,x)^{\frac{1}{n+1}}\, d\vol_{\bd\, 
K}^\lambda(x).
	\end{align*}
	We call $\Omega^\lambda(K)$ the \emph{$\lambda$-floating area} of $K$.
\end{theorem}

Here $\vol_{\bd\, K}^\lambda$ denotes the natural boundary measure with respect to $\Sp^n(\lambda)$ and 
$H_{n-1}^{\lambda}(K,x)$ denotes the (generalized) Gauss-Kronecker curvature on $\bd\, K$, the boundary of $K$, with respect to 
$\Sp^n(\lambda)$ (see Section 3 for details). Furthermore, $\kappa_n$ is the volume of the Euclidean unit Ball 
$B^n_e(0,1)$ in $\R^n$, i.e., $\kappa_n=\vol_n^e(B_e^n(0,1))$.

For $\lambda=0$, i.e.\ Euclidean space, Theorem \ref{thm:intro_main} was first proved in this form by 
Schütt and Werner \cite{Schuett:1990}. For $\lambda=1$, the theorem was established only very recently by the 
authors \cite{Besau:2015a}. 
In this article we now prove the complete form for all $\lambda\in \R$ with a new unifying approach. 
In Section 2 we recall important notions from Euclidean convex geometry. In particular, we investigate the weighted floating 
body. In Section 3 we recall basic facts from hyperbolic geometry. We use the projective Euclidean model and relate hyperbolic convex bodies with Euclidean convex bodies. It is well-known that real space forms admit Euclidean models. We make use of this fact to generalize our results in Subsection 3.2 to real space forms.
The Euclidean models and the results on the weighted floating body are the main tool to prove Theorem \ref{thm:intro_main} in Section 4. 
In Section 5 we investigate the floating area and also the surface area measure of Euclidean convex bodies related to it. In particular, we show the following.

\newpage
\begin{theorem}\label{thm:intro_main2}
	Let $\lambda\in\R$ and $n\in\mathbb{N}$, $n\geq 2$. Then the $\lambda$-floating area $\Omega^\lambda\colon 
\mathcal{K}_0(\Sp^n(\lambda))\to \R$ is
	\begin{enumerate}
		\item[(a)] upper semi-continuous,
		\item[(b)] a valuation, that is, for $K,L\in\mathcal{K}_0(\Sp^n(\lambda))$ such that $K\cup L \in 
\mathcal{K}_0(\Sp^n(\lambda))$ we have that
		\begin{align*}
			\Omega^\lambda(K)+\Omega^\lambda(L) = \Omega^\lambda(K\cup L) + \Omega^\lambda(K\cap L),
		\end{align*}
		\item[(c)] and invariant under isometries of $\Sp^n(\lambda)$. For $\lambda=0$, $\Omega^0$ coincides 
with the affine surface area and is invariant not only under isometries, but all (equi-)affine transformations 
of $\R^n$.
	\end{enumerate}
\end{theorem}

All the properties in Theorem \ref{thm:intro_main2} are well known for the affine surface area, that is, 
$\lambda=0$, see e.g.\ \cites{Schuett:1994,Ludwig:2001,Lutwak:1991,Leichtweiss:1986}. Also, in the spherical 
case, $\lambda=1$, we were able to establish similar results \cite{Besau:2015a}.

Finally in Subsection 5.2 we briefly consider an isoperimetric inequality for the floating area.

\section{The Weighted Floating Body}

In this section we recall the notion of weighted floating bodies introduced in \cite{Werner:2002}. It will 
serve as a unifying framework for dealing with Euclidean, spherical and hyperbolic floating bodies.
In the following we also recall facts from Euclidean convex geometry. For a general reference we refer to 
\cites{Gruber:2007,Gardner:2006,Schneider:2014}. The final goal of this section is to establish Lemma 
\ref{lem:euclapprox}, which is a crucial step in the proof of our main Theorem \ref{thm:intro_main} in Section 
4.

We denote the Euclidean volume by $\vol_n^e$.
If a $\sigma$-finite Borel measure $\mu$ is absolutely continuous to another $\sigma$-finite Borel measure 
$\nu$ on an open set $D\subseteq \R^n$, then we write $\mu \ll_D \nu$. The measure $\mu$ is equivalent to 
$\nu$ on $D$, $\mu\sim_D \nu$, if and only if $\mu\ll_D \nu$ and $\nu\ll_D \mu$.
Evidently, by the Radon--Nikodym Theorem, for a $\sigma$-finite Borel measure $\mu$ we have that $\mu \sim_D 
\vol_n^e$ if and only if there is Borel function $f_\mu\colon D\to \R$ such that $d\mu(x) = f_\mu(x)dx$ and 
$\vol_n^e(\{f_\mu=0\}) = 0$. For a convex body $K\in\K_0(\R^n)$ we consider $\sigma$-finite measures $\mu$ 
such that $\mu \sim_{\inter\, K} \vol_n^e$, where $\inter\, K$ denotes the interior of $K$. Thus, without loss 
of generality, we may assume $\mu$ to be a $\sigma$-finite Borel measure on $\R^n$ with support $K$ and for 
any measurable set $A$ we have
\begin{align*}
	\mu(A) = \int\limits_{A\, \cap\, \inter\, K} f_\mu(x) \, d\vol_n^e(x).
\end{align*}

\begin{definition}[Weighted Floating Body \cite{Werner:2002}]
	Let $K\in\K_0(\R^n)$ and let $\mu$ be a finite non-negative Borel measure on $\inter\, K$ such that $\mu 
\sim_{\inter\, K} \vol_n^e$. For $\delta>0$, we define the \emph{weighted floating body} $\F^\mu_{\delta}\, 
K$, by
	\begin{align*}
		\F^{\mu}_{\delta}\, K = \bigcap \left\{ H^- : \mu(H^+\cap K) \leq \delta^{\frac{n+1}{2}}\right\},
	\end{align*}
	where $H^\pm$ are the closed half-spaces bounded by the hyperplane $H$.
\end{definition}

We will see that the weighted floating body exists (i.e.\ is non-empty) if $\delta$ is small enough. Since it 
is an intersection of closed half-spaces, it is a convex body contained in $K$.

\begin{example}
	For $\mu=\vol_n^e$ we retrieve the Euclidean floating body, denoted by $\F_{\delta}^e\,K$. In the 
literature different normalizations appear. For instance, in \cite{Schuett:1990} the convex floating body is 
defined as $$K_{t} = \bigcap \{ H^-:\vol_n^e(H^+\cap K)\leq t\},$$ which is equivalent to our notion since 
	\begin{align*}
		\F^e_{\delta} \, K = K_{\delta^{(n+1)/2}}.
	\end{align*}
\end{example}

We denote by $\cdot$ the Euclidean scalar product and by $\|.\|$ the Euclidean norm in $\R^n$.
A convex body is uniquely determined by its \emph{support function} $h_K$ defined by
\begin{align*}
	h_K(x) = \max\{ x\cdot y: y\in K\},\quad x\in\R^n.
\end{align*}
The geometric interpretation of the support function is the following:
For a fixed point $x\in\R^n$ and a normal direction $v\in\S^{n-1}$, we denote the hyperplane parallel to the 
hyperplane through $x$ with normal $v$ at distance $\alpha\in\R$ by $H_{x,v,\alpha}$, i.e., 
\begin{align*}
	H_{x,v,\alpha}=\left\{ y\in\R^n : y\cdot v = \alpha + x\cdot v \right\} = H_{0,v,\alpha+x\cdot v}.
\end{align*}
For a given direction $v\in\S^{n-1}$, the support function $h_K(v)$ measures the distance of a supporting 
hyperplane in direction $v$ to the origin. That is, $H_{0,v,h_K(v)}$ is a supporting hyperplane of $K$ in 
direction $v$ and $K$ is given by
\begin{align}\label{eqn:wulffsupport}
	K = \bigcap_{v\in\S^{n-1}} H_{0,v,h_K(v)}^-,
\end{align}
where $H_{x,v,\alpha}^- = \{y\in\R^n: y\cdot v \leq \alpha+x\cdot v\}$.

A closed Euclidean ball of radius $r$ and center $x\in\R^n$ is denoted by $B^n_e(x,r)$. For $K\in\K(\R^n)$ the 
set of points of distance $r$ from $K$ is $B^n_e(K,r)$.
For $K,L\in \K(\R^n)$, the \emph{Hausdorff distance} $\delta^e$ is defined by 
\begin{align*}
	\delta^e(K,L) = \inf \left\{ r\geq 0 : K\subseteq B^n_e(L,r) \text { and } L\subseteq B^n_e(K,r)\right \}.
\end{align*}
Equivalently, we have that
\begin{align*}
	\delta^e(K,L) = \sup_{v\in\S^{n-1}} \left\| h_K(v)-h_L(v)\right\|.
\end{align*}

Given a continuous function $f\colon \S^{n-1}\to \R$ the \emph{Wulff shape} $[f]$ (also called Aleksandrov 
body, see \cite{Gardner:2002}*{Sec.\ 6}) of $f$ is, unless it is the empty set, the convex body defined by
\begin{align}\label{eqn:defwulff}
	[f] = \bigcap_{v\in\S^{n-1}} H_{0,v,f(v)}^-.
\end{align}
For a positive continuous function $f$ the Wulff shape is a convex body containing the origin in its interior. 
For a convex body $K$ we have $K=[h_K]$, i.e., the Wulff shape associated with 
$h_K$ is $K$ itself. The concept of Wulff shapes has many applications, see e.g.\ \cite{Schneider:2014}*{Sec.\ 
7.5} for a short exposition.

The weighted floating body is a Wulff shape.
\begin{proposition}\label{prop:euclpar}
	Let $K\in\K_0(\R^n)$, $\mu$ be a finite non-negative Borel measure on $\inter \, K$ such that 
$\mu\sim_{\inter\, K} \vol_n^e$ and $\delta\in \left(0,\mu(K)^{\frac{2}{n+1}}\right)$. For $v\in\S^{n-1}$, 
there exists a unique $s_{\delta}(v)\in \R$ determined by
	\begin{align*}
		\mu\left(K\cap H^+_{0,v,h_K(v)-s_{\delta}(v)}\right) = \delta^{\frac{n+1}{2}}.
	\end{align*}
	In particular, $s_{\delta}(v)=s(\delta,v)$ is continuous on $\left(0,\mu(K)^{\frac{2}{n+1}}\right)\times 
\S^{n-1}$ and strictly increasing in $\delta$.
	Moreover, the weighted floating body $\F_{\delta}^\mu\, K$ exists if and only if the Wulff shape 
$[h_K-s_{\delta}]$ exists and in this case we have that
	\begin{align}\label{eqn:parfloat}
		\F^\mu_{\delta}\, K = [h_K-s_{\delta}].
	\end{align}
\end{proposition}
\begin{proof}
	We consider $G\colon \S^{n-1}\times \R \to [0,\mu(K)]$ defined by
	\begin{align*}
		G(v,\Delta) = \mu\left(K\cap H^+_{0,v,h_K(v)-\Delta}\right).
	\end{align*}
	Since $\mu$ is a non-negative Borel measure equivalent to $\vol_n$, we can find a Borel function 
$f_\mu\colon \R^n \to [0,\infty)$ such that $f_\mu>0$ almost everywhere on $\inter\, K$ and $f_\mu=0$ else. We 
can therefore write 
	\begin{align*}
		G(v,\Delta) = \int\limits_{H^+_{0,v,h_K(v)-\Delta}} f_\mu(x)\, dx = 
\int\limits_{h_K(v)-\Delta}^{h_K(v)} \int\limits_{v^\bot} f_\mu(w+tv)\, dw\, dt.
	\end{align*}
	For the second equality we used Fubini's theorem and the substitution $x=w+tv$, where $w\in 
v^\bot=\{y\in\R^n : y\cdot v = 0\}$ and $t\in\R$ are uniquely determined by $x$.
	Thus $G$ is strictly increasing in $\Delta$ for $\Delta\in \left(0,h_K(v)+h_K(-v)\right)$ from $0$ to 
$\mu(K)$.
	To see that $G$ is continuous, first note that $K(v,\Delta):=K\cap H_{0,v,h_K(v)-\Delta}^+$ depends 
continuously on $(v,\Delta)\in\S^{n-1} \times \R$ with respect to the Hausdorff distance $\delta^e$. This 
follows, since $h_K$ is continuous and the map $(v,\lambda)\to H^+_{0,v,\lambda}\cap K$ is continuous in 
$v\in\S^{n-1}$ and $\lambda\in \R$. Now, since $\vol_n$ is continuous on $\K_0(\R^n)$, see 
\cite{Schneider:2014}*{Thm.\ 1.8.20}, and since $\mu\sim_{\inter\, K}\vol_n$, we conclude that $\mu$ is 
continuous on $\K_0(\R^n)\cap K$ and therefore $\mu(K(v,\Delta))=G(v,\Delta)$ is continuous in $v$.
	
	Hence, for $\delta\in \left(0,\mu(K)^{\frac{2}{n+1}}\right)$ there is a unique $s_{\delta}(v)\in 
(0,h_K(v)+h_K(-v))$ such that
	\begin{align*}
		\delta^{\frac{n+1}{2}} = G(v,s_{\delta}(v)),
	\end{align*}
	which is strictly increasing in $\delta$ and continuous.
	
	To prove $(\ref{eqn:parfloat})$, we first consider a fixed $v\in\S^{n-1}$. For $t_1<t_2$, we have that 
$H^+_{0,v,h_K(v)-t_1} \subseteq H^+_{0,v,h_K(v)-t_2}$. The maximal $t$ such that $G(v,t)\leq 
\delta^{\frac{n+1}{2}}$ is $t=s_{\delta}(v)$. Hence, we have that
	\begin{align*}
		\bigcap \left\{ H^-_{0,v,h_K(v)-t} : t\in\R\text{ such that } G(v,t)\leq \delta^{\frac{n+1}{2}}\right 
\} = H^-_{0,v,h_K(v)-s_{\delta}(v)}.
	\end{align*}
	Finally, we conclude that
	\begin{align*}
		\F^{\mu}_{\delta}\, K &= \bigcap \left\{ H^-_{0,v,h_K(v)-t}:v\in\S^{n-1}, t\in\R \text{ such that } 
G(v,t)\leq \delta^{\frac{n+1}{2}}\right\}\\
		&= \bigcap_{v\in\S^{n-1}} H^-_{0,v,h_K(v)-s_{\delta}(v)} = [h_K-s_{\delta}].\hfill \qedhere
	\end{align*}
\end{proof}

For a convex body $K\in\mathcal{K}_0(\R^n)$ and a boundary point $x\in\bd\, K$ we define the set of normal 
vectors $\sigma(K,x)$ of $K$ in $x$, also called the \emph{spherical image} of $K$ at $x$ (see 
\cite{Schneider:2014}*{p.\ 88}), by
\begin{align*}
	\sigma(K,x)=\{v\in\S^{n-1}: H_{x,v,0} \text{ is a supporting hyperplane to $K$ in $x$}\}.
\end{align*}

A boundary point $x$ is called \emph{regular} if $\sigma(K,x)$ is a single point, that is, $K$ has a unique 
outer unit normal vector $N_x$ at $x$. Note that for a convex body almost all boundary points are regular (see 
\cite{Schneider:2014}*{Thm.\ 2.2.5}). A boundary point $x$ is \emph{exposed} if and only if there is a support 
hyperplane $H$ such that $K\cap H = \{x\}$.

A subset $S\subseteq \S^{n-1}$ is a \emph{spherical convex body} if and only if the positive hull 
$\mathrm{pos}\, S=\{\lambda s: \lambda\geq 0, s\in S\}$ is a closed convex cone in $\R^n$. The spherical image 
at a boundary point $x$ is a spherical convex body and the closed convex cone generate by it is the 
\emph{normal cone} $N(K,x)=\mathrm{pos}\, \sigma(K,x)$.
If $K$ has non-empty interior, then $\sigma(K,x)$ is proper for any boundary point, that is, the normal cone 
does not contain any linear subspace. 

The \emph{spherical Hausdorff distance} $\delta^s$ is a metric on spherical convex bodies induced by the 
spherical distance 
\begin{align}\label{def:sdis}
	d_s(x,y)=\arccos(x\cdot y)
\end{align}
on $\S^{n-1}$ in the following way: For a subset $A\subset\S^{n-1}$ we denote by $A_\varepsilon$ the 
$\varepsilon$-neighborhood of $A$, i.e., $A_\varepsilon = \{a\in\S^{n-1}: d_s(a,A)<\varepsilon\}$. Then, for 
spherical convex bodies $S$ and $T$, we have that
\begin{align*}
	\delta_s(S,T) = \inf\{ \varepsilon\geq 0 : S\subseteq T_{\varepsilon} \text{ and } S\subseteq 
T_{\varepsilon}\}.
\end{align*}

The spherical Hausdorff distance induces a metric on the closed convex cones with apex at the origin via the 
positive hull $\mathrm{pos}$.
Hence, we say that a sequence of closed convex cones $C_i$ converges to a closed convex cone $C$ if and only 
if the sequence 
of spherical convex bodies $C_i\cap \S^{n-1}$ converges to $C\cap \S^{n-1}$ with respect to $\delta^s$. 

Fix $z\in\S^{n-1}$ and let $(C_i)_{i=1}^{\infty}$ be a sequence of closed convex cones contained in the open 
half-space $\inter\,H^+_{0,z,0}$. Then $C_i$ converges to a closed convex cone $C$ contained in the same 
open-half space with respect to $\delta^s$ if and only if the sections of the convex cones with the affine 
hyperplane $H_{0,z,1}$ converge with respect to the Euclidean Hausdorff metric $\delta^e$ in $H_{0,z,1}$. 

We define the \emph{gnomonic projection} $g_z\colon \inter\, H^+_{0,z,0}\to z^\bot\cong \R^{n-1}$ by
\begin{align}\label{def:gnomonic}
	g_z(x) = (x\cdot z)^{-1} x - z.
\end{align}
Then $g_z$ maps closed convex cones contained in the open half-space $\inter\, H^+_{0,z,0}$ to convex bodies 
in $z^\bot\cong\R^{n-1}$. By the previous statement we find, that the gnomonic projection induces an 
homeomorpism between the space of closed convex cones in $\inter\, H^+_{0,z,0}$ with respect to the $\delta^s$ 
and the space of convex bodies in $z^\bot\cong \R^{n-1}$ with respect to $\delta^e$. Compare also 
\cite{Besau:2015}*{Cor.\ 4.5}.

\bigskip
By Proposition \ref{prop:euclpar}, $s_{\delta}(v)=s(\delta,v)$ is continuous as a function in $(\delta,v)$. It 
converges point-wise to $0$ as $\delta\to 0^+$. By the compactness of $\S^{n-1}$, we have that 
$s_{\delta}(.)$ converges uniformly to $0$ as $\delta\to 0^+$.
This implies the convergence of $[h_K-s_{\delta}]$ to $[h_K]=K$, see e.g.\ \cite{Schneider:2014}*{Lem.\ 
7.5.2}. We conclude
\begin{align}\label{eqn:conv_float}
	\lim_{\delta\to 0^+} \F_{\delta}^\mu\, K = K.
\end{align}

Our next goal is to show that the convergence of the weighted floating body is locally determined. This fact 
and therefore most of the following lemmas are probably known for the most part. However, since we were only 
able to find references in particular cases, for instance see e.g.\ \cite{Schuett:2003} for related results, 
and also for the convenience of the reader, we include proofs for the following.

We consider a regular boundary point $x\in\bd\, K$ and investigate the behavior of $\F_\delta^\mu\, K$ near 
$x$ for $\delta\to 0^+$. The shape of $\F_\delta^\mu\, K$ near $x$ is determined by a neighborhood of 
directions of the unique normal $N_x$ of $K$ at $x$. For $s<t$, we have that $\F_s^\mu\, K\supseteq \F_t^\mu\, 
K$. In particular, if $0\in\inter\, \F_t^\mu\, K$, then for all $\delta\in (0,t)$ we have $0\in\inter\, 
\F_\delta^\mu\, K$. In this case we define $x_\delta^K$ as the unique intersection point of 
$\bd\,\F_\delta^\mu\, K$ with the ray $\mathrm{pos}\,\{x\}$. Hence, $\lim_{\delta\to 0^+} x_\delta^K = x$. 
We use $x_\delta$ to control the limit process $\F_\delta^\mu\, K \to K$ near $x$ as 
$\delta\to 0^+$.

The first step is to consider a convergent sequence of Wulff shapes $[f_i]\to[f]$, where 
$(f_i)_{i\in\mathbb{N}}$ and $f$ are positive continuous functions on $\S^{n-1}$. Thus $0\in\inter\, [f_i]$ 
and $0\in\inter\,[f]$. We show that, for any regular boundary point $x\in\bd\, [f]$ and any neighborhood of 
directions around the normal $N_x$ of $[f]$ at $x$, there is $i_0\in\mathbb{N}$ such that, for all $i>i_0$, 
$x_i:=\bd\,[f_i]\cap\mathrm{pos}\{x\}$ is determined by the values of $f_i$ in that neighborhood.

\begin{lemma}[Local dependence of a convergent sequence of Wulff shapes]\label{lem:locality}
	Let $f_i\colon\S^{n-1}\to (0,\infty)$, $i\in\mathbb{N}$, be a sequence of positive continuous function 
uniformly convergent to $f\colon\S^{n-1}\to (0,\infty)$. Then for $x\in\mathrm{reg}\, [f]$ and $\varepsilon>0$ 
there exists $i_0\in\mathbb{N}$ such that, for all $i>i_0$, we have that
	\begin{align*}
		[f_i]\cap \mathrm{pos}\{x\} = \bigcap \left\{ H^-_{0,v,f_i(v)}: d_s(v,N_x) < 
\varepsilon\right\}\cap\mathrm{pos}\{x\},
	\end{align*}
	where $N_x$ is the unique outer unit normal of $[f]$ at $x$.
\end{lemma}

Since the weighted floating body can be viewed as a Wulff shape and converges to $K$ as $\delta\to 0^+$ we 
obtain the following corollary.
\begin{corollary}[Locality of the weighted floating body]\label{cor:locality}
	Let $K\in\K_0(\R^n)$ be such that $0\in\inter\, K$. Then for $x\in\mathrm{reg}\, K$ and $\varepsilon>0$ 
there exists $\delta_\varepsilon>0$ such that for all $\delta <\delta_\varepsilon$, we have $0\in\inter\, 
\F^\mu_\delta\, K$ and 
	\begin{align}\notag
		\mathrm{conv}(x_\delta^K,0) &= \F_\delta^\mu \, K \cap\mathrm{pos}\,\{x\}\\\label{eqn:locality}
		&= \bigcap\left\{ H^-_{0,v,h_K(v)-s_\delta(v)}: d_s(v,N_x)< \varepsilon\right\} \cap 
\mathrm{pos}\,\{x\}.
	\end{align}
	where $s_\delta(v)$ is uniquely determined by
	\begin{align*}
		\delta^{\frac{n+1}{2}} = \mu\left(K\cap H^+_{0,v,h_K(v)-s_\delta(v)}\right).
	\end{align*}
\end{corollary}
Before we prove Lemma \ref{lem:locality}, we recall some common notation.
For $u\in\S^{n-1}$,
	$F(K,u) = K\cap H_{0,u,h_K(u)}$ 
is the \emph{exposed face} of $K$ in direction $u$. 
The following is an easy observation.
\begin{lemma}[Convergence of exposed faces]\label{lem:facecon}
	Let $K_i\to K$ in $\K_0(\R^n)$ with respect to the Hausdorff distance $\delta^e$. If $u\in\S^{n-1}$ such 
that $F(K,u)=\{x\}$ is an exposed point of $K$, then $F(K_i,u)\to \{x\}$.
\end{lemma}
\begin{proof}
	Since $K_i$ converges to $K$ with respect to $\delta^e$, any sequence $x_i\in F(K_i,u)\subseteq K_i$ has a 
convergent subsequence with limit $y\in K$, see e.g.\ \cite{Schneider:2014}*{Thm.\ 1.8.7}. 
	Let $R>0$ be such that $K\cup\bigcup_{i\in\mathbb{N}}K_i\subseteq B^n_e(0,R)$. Then $x_i\in F(K_i,u) 
\subseteq H_{0,u,h(K_i,u)}\cap B^n_e(0,R)$ and also $H_{0,u,h(K_i,u)}\cap B^n_e(0,R) \to H_{0,u,h(K,u)}\cap 
B^n_e(0,R)$. Hence, for the limit point $y$ of the convergent subsequence, we also have $y\in H_{0,u,h(K,u)}$ 
and therefore $y\in K\cap H_{0,u,h(K,u)} = F(K,u) = \{x\}$.
\end{proof}

We denote the set of convex bodies with $0$ in the interior by $\K_{00}(\R^n)$.
For $K\in\K_{00}(\R^n)$, $K^\circ = \{y\in\R^n:x\cdot y \leq 1\}$ is the \emph{polar body} of $K$. 
For $x\in\bd\, K$, set $\widehat{x} = \{y\in K^\circ:x\cdot y = 1\}$. Then $N(K,x) = \mathrm{pos}\, 
\widehat{x}$, see \cite{Schneider:2014}*{Lem.\ 2.2.3}.
For $x\in\bd\, K$, we have $h_{K^\circ}\left(x/\|x\|\right) = 1/\|x\|$. Hence, $\widehat{x} = 
F\left(K^\circ,x/\|x\|\right)$, 
or equivalently
\begin{align}\label{eqn:normalcone}
	N(K,x) = \mathrm{pos}\, F\left(K^\circ, x / \|x\|\right).
\end{align}
For a proof of the following fact see, e.g., \cite{Groemer:1996}*{Lem.\ 2.3.2}.
\begin{lemma}[Continuity of the polar map]\label{lem:polarcont}
	Let $(K_i)_{i\in\mathbb{N}}$ be a sequence in $\K_{00}(\R^n)$ converging to $K\in K_{00}(\R^n)$ with respect 
to the Hausdorff distance $\delta^e$. Then also $K_i^\circ \to K^\circ$.
\end{lemma}

By (\ref{eqn:normalcone}), the normal cone $N(K,x)$ at a boundary point $x$ is related to the exposed face of 
the polar body $K^\circ$ in direction $x/\|x\|$. Using the continuity of the polar map, Lemma 
\ref{lem:polarcont}, and the convergence of the exposed faces, Lemma \ref{lem:facecon}, we now obtain the 
convergence of the normal cones in regular boundary points. Note that for a regular boundary point $x\in 
\mathrm{reg}\, K$, $N_x/(N_x\cdot x)$ is an exposed point of $K^\circ$, i.e., $F\left(K^\circ,x/\|x\|\right) = 
\left\{N_x/(N_x\cdot x)\right\}$.

\begin{lemma}[Convergence of the normal cone]\label{lem:norconecon}
	Let $f_i$ be a sequence of positive continuous functions on $\S^{n-1}$, uniformly convergent to a positive 
continuous function $f$. For $x\in\mathrm{reg}\, [f]$ we set $\{x_i\} =\mathrm{pos}\{x\}\cap \bd\, [f_i]$. 
Then
	\begin{align*}
		\lim_{i\to\infty} N([f_i],x_i) = N([f],x).
	\end{align*}
	In particular, we have that
	\begin{align*}
		\lim_{i\to\infty} \sigma([f_i], x_i) = \lim_{i\to\infty} N([f_i], x_i)\cap\S^{n-1} = 
N([f],x)\cap\S^{n-1} = \{N_x\}.
	\end{align*}
\end{lemma}
\begin{proof}
	We set $z=x/\|x\| = x_i/\|x_i\|$. By (\ref{eqn:normalcone}),
	\begin{align*}
		\mathrm{pos}\,\{N_x\} = N([f],x) = \mathrm{pos}\, F([f]^\circ,z),
	\end{align*}
	or equivalently $F([f]^\circ,z) = \{N_x/(N_x\cdot x)\}$.
	With Lemma \ref{lem:polarcont} and Lemma \ref{lem:facecon}, we conclude that
	\begin{align}\label{eqn:convnorm1}
		\lim_{i\to \infty} F([f_i]^\circ,z) = \left\{N_x/(x\cdot N_x)\right\}.
	\end{align}
	Since $0\in\inter\, [f_i]^\circ$ the exposed face in direction $z$ has a positive distance $a_i$ from the 
origin. Therefore $a_i=h_{F([f_i]^\circ,z)}(z)>0$, $F([f_i]^\circ,z)\subseteq H_{0,z,a_i}$. Also 
$h_{F([f]^\circ,z)}(z)=1/\|x\|>0$ and $\lim_{i\to \infty} a_{i} = 1/\|x\|$. Hence, the convex cone generated 
by the exposed face does not contain any linear subspace, or equivalently, 
$S_i:=(\mathrm{pos}\,F([f_i]^\circ,z)) \cap \S^{n-1}$ as well as $(\mathrm{pos}\, F([f]^\circ,z))\cap 
\S^{n-1}=\{N_x\}$ are contained in the open hemisphere with center in $z$. 
	
	We have to show that $S_i$ converges to $\{N_x\}$ with respect to spherical Hausdorff distance $\delta^s$. 
This is equivalent to the convergence of $g_z(S_i)$ to $g_z(N_x)$ in $z^\bot$ with respect to the Euclidean 
Hausdorff distance. Here, $g_z$ is the gnomonic projection in $z$, see (\ref{def:gnomonic}). We obtain 
$g_z(S_i) = (1/a_i)F([f_i]^\circ,z)-z$ and $g_z(N_x) = N_x/(N_x\cdot z)-z$. Since $a_i\to 1/\|x\|$ and 
$F([f_i]^\circ,z) \to F([f]^\circ,z) = \{N_x/(N_x\cdot x)\}$, we conclude that $g_z(S_i) \to g_z(N_x)$. This 
yields that $S_i\to \{N_x\}$, or equivalently, the convex cones generated by the exposed faces converge, 
i.e., 
	\begin{align*}
		\lim_{i\to \infty} \mathrm{pos}\, F([f_i]^\circ,z) 
		= \lim_{i\to \infty} \mathrm{pos}\, S_i 
		= \mathrm{pos}\, \{N_x\} 
		= \mathrm{pos}\, F([f]^\circ,z).
	\end{align*}
	By $(\ref{eqn:normalcone})$, this concludes the proof.
\end{proof}

We are now ready to prove Lemma \ref{lem:locality}.
\begin{proof}[\indent Proof of Lemma \ref{lem:locality}]
	Assume the opposite. Then there exists $x\in\mathrm{reg}\, [f]$ and $\varepsilon>0$ such that, for all 
$i\in\mathbb{N}$, we have
	\begin{align*}
		\left(\bigcap \left\{H^-_{0,v,f_i(v)}: v\in\S^{n-1}, d_s(v,N_x)< 
\varepsilon\right\}\cap\mathrm{pos}\,\{x\}\right)\backslash \left([f_i] \cap \mathrm{pos}\,\{x\}\right)  \neq 
\emptyset.
	\end{align*}
	By definition, $[f_i] = \{y\in\R^n: y\cdot v\leq f_i(v)\, \forall v\in\S^{n-1}\}$, Therefore, for $\{x_i\} 
= \bd\, [f_i] \cap \mathrm{pos}\,\{x\}$, there exists $z_i\in\S^{n-1}$ such that $x_i\cdot z_i = f_i(z_i)$. 
This yields $z_i\in\sigma([f_i],x_i)$ and we conclude
	\begin{align*}
		[f_i]\cap \mathrm{pos}\,\{x\} 
		&= \mathrm{conv}(x_i,0)= H^-_{x_i,z_i,0} \cap \mathrm{pos}\,\{x\}\\
		&= H^-_{0,z_i,z_i\cdot x_i} \cap \mathrm{pos}\,\{x\} 
		= H^-_{0,z_i,f_i(z_i)} \cap \mathrm{pos}\,\{x\}.
	\end{align*}
	Thus $d_s(z_i,N_x)\geq \varepsilon >0$ for all $i$. By compactness of $\S^{n-1}$, there is a convergent 
subsequence of $(z_i)_{i\in\mathbb{N}}$ with limit $z\neq N_x$. This is a contradiction, since 
$\sigma([f_i],x_i)\to \sigma([f],x)=\{N_x\}$ by Lemma \ref{lem:norconecon}.
\end{proof}

By Corollary \ref{cor:locality}, the weighted floating body is locally determined near any regular boundary 
point $x$. If is $x$ also exposed, then a neighborhood of $x$ in $K$ already determines the shape of 
$\F_\delta^\mu\, K$ near $x$ for $\delta\to 0^+$.

\begin{lemma}[Approximation of the weighted floating body]\label{lem:euclapprox}
	Let $K\in\K_0(\R^n)$ and $x\in\bd\, K$ be a regular and exposed point, that is, there is a unique outer 
unit normal $N_x$ and $K\cap H_{x,N_x,0} = \{x\}$. For $\varepsilon >0$ set $K'=K\cap B^n_e(x,\varepsilon)$. 
Then $x\in\bd\, K'$ is a regular and exposed point of $K'$. Furthermore:
	\begin{enumerate}
		\item[(i)] There exists $\Delta_\varepsilon$ such that for all $\Delta < \Delta_\varepsilon$ we have
		\begin{align*}
			K'\cap H^+_{x,N_x,-\Delta} = K\cap H^+_{x,N_x,-\Delta}.
		\end{align*}
		\item[(ii)] There exists $\xi_\varepsilon$ and $\eta_\varepsilon$ such that, for all $v\in \S^{n-1}$ 
with $d_s(v,N_x)<\xi_\varepsilon$ and $\Delta <\eta_\varepsilon$, we have
		\begin{align*}
			K'\cap H^+_{x,v,-\Delta} = K\cap H^+_{x,v,-\Delta}
		\end{align*}
		\item[(iii)] Let $0\in\inter\, K'$. There exists $\delta_\varepsilon$ such that, for all 
$\delta<\delta_\varepsilon$, we have 
		$0\in \inter(\F^\mu_\delta\, K'\cap \F^\mu_\delta\, K)$ and $x_\delta^{K'} = x_\delta^K$, where 
		$\{x_\delta^{K^*}\} = \bd\, \F^\mu_\delta\, K^* \cap \mathrm{pos}\,\{x\}$.
	\end{enumerate}
\end{lemma}
\begin{proof}
	\emph{(i):} Assume that the statement is false. Then there exists $\varepsilon>0$ such that for all 
$\Delta>0$, we have
	\begin{align*}
		\emptyset 
		&\neq \left(K\cap H^+_{x,N_x,-\Delta}\right) \backslash \left(K'\cap H^+_{x,N_x,-\Delta}\right)
		= (K\backslash B^n_e(x,\varepsilon)) \cap H^+_{x,N_x,-\Delta} \\
		&\subseteq (K\backslash \inter\, B^n_e(x,\varepsilon)) \cap H^+_{x,N_x,-\Delta}.
	\end{align*}
	For $\Delta_1\leq \Delta_2$, we have $$(K\backslash \inter B^n_e(x,\varepsilon))\cap H^+_{x,N_x,-\Delta_1} 
\subseteq (K\backslash \inter\, B^n_e(x,\varepsilon))\cap H^+_{x,N_x,-\Delta_2}.$$ 
	By compactness, we conclude that $\emptyset \neq (K\backslash \inter\, B^n_e(x,\varepsilon)) \cap 
H^+_{x,N_x,0}$. This is a contradiction, since $K\cap H^+_{x,N_x,0} = \{x\}$.
	
	\bigskip
	\emph{(ii):} Let $\varepsilon >0$. By (i), there exists $\Delta_{\varepsilon/2}$ such that 
$\Delta_{\varepsilon/2}<\varepsilon/2$ and 
	\begin{align}\label{eqn:contcon3}
		(K\backslash \inter\, B^n_e(x,\varepsilon/2))\cap H^+_{x,N_x,-\Delta_{\varepsilon/2}} = \emptyset.
	\end{align}
	We set $\eta_\varepsilon=(1/2)\Delta_{\varepsilon/2}$ and
	\begin{align*}
		\xi_{\varepsilon} = \sqrt{2+(2/\varepsilon)\Delta_{\varepsilon/2}} - 
\sqrt{2+(1/\varepsilon)\Delta_{\varepsilon/2}}.
	\end{align*}
	Then $\eta_\varepsilon<\varepsilon/4$ and $\xi_\varepsilon>0$.
	We show that for all $v\in\S^{n-1}$ with $d_s(v,N_x)<\xi_{\varepsilon}$, we have that
	\begin{align}\label{eqn:contcon2}
		(K\backslash \inter\, B^n_e(x,\varepsilon)) \cap H^+_{x,v,-\eta_{\varepsilon}} =\emptyset,
	\end{align}
	which yields $K'\cap H^+_{x,v,-\Delta} = K\cap H^+_{x,v,-\Delta}$ for all $\Delta < \eta_{\varepsilon}$.

	Assume that $(\ref{eqn:contcon2})$ is not true. Then there exists $z\in (K\backslash \inter\, 
B^n_e(x,\varepsilon)) \cap H^+_{x,v,-\eta_\varepsilon}$. 
	Since $K$ is convex, the segment $\mathrm{conv}(z,x)$ is contained in $K$. Furthermore, since 
$\|z-x\|>\varepsilon$, there exists $z'\in\bd\, B^n_e(x,\varepsilon)\cap \mathrm{conv}(z,x)$. 
	
	We will show that $z'\in H^+_{x,N_x,-\Delta_{\varepsilon/2}}$, i.e., $z'\cdot N_x \geq x\cdot N_x - 
\Delta_{\varepsilon/2}$. This will be a contradiction to $(\ref{eqn:contcon3})$, since $z'\in K\backslash 
\inter\, B^n_e(x,\varepsilon/2)$.
	
	Since $z'\in\bd\, B^n_e(x,\varepsilon)$, we have $z'=x + \varepsilon \|z'-x\|^{-1}(z'-x)$ and therefore
	$z'\cdot N_x = x\cdot N_x + \varepsilon \|z'-x\|^{-1} (z'-x)\cdot N_x$.
	Since $z'\in H^+_{x,v,-\eta_\varepsilon}$, we obtain $z'\cdot v \geq x\cdot v - \eta_{\varepsilon}$ or
	$\|z'-x\|^{-1}(z'-x)\cdot v \geq -\varepsilon^{-1}\eta_{\varepsilon}$.
	Put $w= \|z'-x\|^{-1}(z'-x)$. Then
	\begin{align*}
		\|w-v\|^2 = 2-2(w\cdot v) \leq 2\left(1+\varepsilon^{-1}\eta_{\varepsilon}\right).
	\end{align*}
	Note that $d_s(v,N_x)\leq \xi_\varepsilon$ implies that $\|v-N_x\|\leq \xi_\varepsilon$. This, together 
with the definition of $\xi_\varepsilon$, implies that
	\begin{align*}
		w\cdot N_x 
		&= 1-\frac{\|w-N_x\|^2}{2} \geq 1-\frac{(\|w-v\|+\|v-N_x\|)^2}{2}\\
		& = w\cdot v - \|w-v\|\|v-N_x\| - \frac{\|v-N_x\|^2}{2}\\
		&\geq 
-\frac{\eta_\varepsilon}{\varepsilon}-\sqrt{2+\frac{2\eta_\varepsilon}{\varepsilon}}\xi_\varepsilon - 
\frac{\xi_\varepsilon^2}{2}
		=1 - \left(\sqrt{1+\frac{\eta_\varepsilon}{\varepsilon}} + \frac{\xi_\varepsilon}{\sqrt{2}}\right)^2
		= -\frac{\Delta_{\varepsilon/2}}{\varepsilon}.
	\end{align*}
	Hence $z'\cdot N_x \geq x\cdot N_x - \Delta_{\varepsilon/2}$.
	
	\bigskip
	\emph{(iii):} By Proposition \ref{prop:euclpar}, we can write $\F_{\delta}^\mu\, K' = 
\bigcap_{v\in\S^{n-1}} H^-_{x,v,-s^{K'}(\delta,v)}$.
	Here $-s^{K'}(\delta,v)$ is uniquely determined by $\delta^{(n+1)/2}=\mu\left(K'\cap 
H^+_{x,v,-s^{K'}(\delta,v)}\right)$ and is continuous in both arguments.
	By (ii), there exists $\xi_\varepsilon$ and $\eta_\varepsilon$ such that for all $v\in\S^{n-1}$ with 
$d_s(v,N^K_x)<\xi_{\varepsilon}$, we have 
	\begin{align*}
		\mu\left(K'\cap H^+_{x,v,-\Delta}\right) = \mu\left(K\cap H^+_{x,v,-\Delta}\right),
	\end{align*}
	for all $\Delta < \eta_\varepsilon$. Hence there exists $\delta_1>0$ such that for all $\delta<\delta_1$ 
and $d_s(v,N_x^K)<\xi_\varepsilon$ we have $s^K(\delta,v)= s^{K'}(\delta,v)$.

	By Corollary \ref{cor:locality} applied to $K$ and $K'$ with $\varepsilon=\xi_\varepsilon$ there exist 
$\delta_2$ and $\delta_3$ such that for all $\delta<\min\{\delta_1,\delta_2,\delta_3\}$ we have $0\in\inter\, 
\F^\mu_\delta\, K'$, $0\in\inter\, \F^\mu_\delta\, K$ and
	\begin{align*}
		(\F_\delta^\mu K) \cap\mathrm{pos}\,\{x\} &= \bigcap\left\{ H^-_{x,v,-s^K(\delta,v)}: v\in \S^{n-1}, 
d_s(v,N_x^K)<\xi_\varepsilon\right\} \cap \mathrm{pos}\,\{x\}\\
		&= \bigcap\left\{ H^-_{x,v,-s^{K'}(\delta,v)}:v\in\S^{n-1},d_s(v,N_x^K)<\xi_\varepsilon\right\} \cap 
\mathrm{pos}\,\{x\}\\
		&= (\F_{\delta}^\mu\, K')\cap \mathrm{pos}\,\{x\}.
	\end{align*}
	This implies in particular that $x_\delta^{K'} = x_\delta^{K}$ since it is the unique intersection point 
of $\mathrm{pos}\, \{x\}$ with the boundary of the floating body $\F_\delta^\mu K'$, or $\F_\delta^\mu K$.
\end{proof}

\section{Hyperbolic Convex Geometry}

In the theory of Riemannian manifolds, hyperbolic $n$-space $\H^n$ is the simply-connected, complete 
Riemannian manifold of constant sectional curvature $-1$. Hyperbolic convex bodies are compact subsets such 
that for any two points in the set, the geodesic segment between them is contained in the set. Hyperbolic 
convex geometry is the study of intrinsic notions of hyperbolic convex bodies.

In his famous Erlangen program Felix Klein characterized geometries based on their symmetry groups. In the 
spirit of this approach, we may view Euclidean convex geometry as the study of notions on Euclidean convex 
bodies that are invariant under the group of rigid motions.
In the projective model (also known as Beltrami--Cayley--Klein model) of hyperbolic space, that is, in the 
open unit ball $\B^n$, hyperbolic convex geometry can be viewed as the study of notions on Euclidean convex 
bodies $K\in\K(\B^n)$, invariant under hyperbolic motions.

In the following we recall basic facts about the projective model of hyperbolic space. For a rigorous 
exposition see, e.g., \cite{Alekseevskij:1993} or \cite{Ratcliffe:2006}.

We consider $\B^n$ together with the Riemannian metric tensor $g^h$ which defines a scalar product in tangent 
space $T_p\B^n$ for any point $p\in\B^n$ by
\begin{align*}
	g^h_p(X_p,Y_p) = \frac{X_p\cdot Y_p}{1-\|p\|^2} + \frac{\left(X_p\cdot p\right)\left(Y_p\cdot 
p\right)}{\left(1-\|p\|^2\right)^2},\quad X_p,Y_p\in T_p\B^n.
\end{align*}
Here and in the following we use the natural identification of $T_p\B^n=T_p\R^n$ with $\R^n$.
Then $(\B^n,g^h)$ is a simply-connected, complete Riemannian manifold with constant sectional curvature $-1$ 
and therefore isometric to $\H^n$.

The Euclidean metric tensor is $g^e$ and it is induced naturally by $g^e_p(X_p,Y_p) = X_p\cdot Y_p$ for 
$X_p,Y_p\in T_p\R^n\cong \R^n$.
When $p=0$, then 
\begin{align}\label{eqn:tensor_equal}
	g^h(X_0,Y_0) = X_0\cdot Y_0 = g^e(X_0,Y_0)
\end{align} 
and therefore the Euclidean metric tensor at the origin agrees with the hyperbolic metric tensor.
Geodesic curves in $(\B^n,g^h)$ are straight lines in $\R^n$ intersected with $\B^n$ and the geodesic 
distance,
or \emph{hyperbolic distance}, $d_h(p,q)$ between $p,q\in \B^n$ is, see for example 
\cite{Alekseevskij:1993}*{Sec.\ 1.5} and \cite{Ratcliffe:2006}*{Ch.\ 6},
\begin{align}\label{eqn:hypdis}
	\cosh\, d_h(p,q) = \frac{1-p\cdot q}{\sqrt{1-\|p\|^2}\sqrt{1-\|q\|^2}}.
\end{align}
Note that 
\begin{equation}\label{eqn:dh}
\tanh\, d_h(p,0) = \|p\|.
\end{equation}

Isometries of the projective model are also called motions and the \emph{group of motions} is $\M(\B^n)$.
The group of motions $\M(\B^n)$ is isomorphic to the restricted Lorentz group $\mathrm{SO}^+(n,1)$ and 
hyperbolic $n$-space is characterized by $\M(\B^n)$ in the following sense: the homogeneous space defined by 
$(\mathrm{SO}^+(n,1),\mathrm{SO}(n))$ is isomorphic to $\H^n$, see e.g.\ \cite{Alekseevskij:1993}*{Ch.\ 1, 
\textsection 2}. Note that, in the projective model, a hyperbolic motion extends to a uniquely determined 
collineation of the projective closure of $\R^n$ and conversely any collineation that maps $\B^n$ to $\B^n$ 
restricts to an hyperbolic motion on $\B^n$.

Geodesics in $(\B^n,g^h)$ are the chords of $\B^n$. More general, any totally geodesic subspace of dimension 
$k$, called a \emph{$k$-plane} of $(\B^n,g^h)$, is the intersection of an affine subspace of $\R^n$ with 
$\B^n$.
Hence, a line is a $1$-plane or chord of $\B^n$ and a hyperplane is a $(n-1)$-plane. 

The (hyperbolic) exponential map $\mathrm{exp}_p^h$ in a point $p\in\B^n$ maps any tangent vector $X_p\in 
T_p\B^n$ to the uniquely determined point $q$ in $\B^n$ such that $d_h(p,q) = \|X_p\|$. For the unit speed 
geodesic path $\gamma\colon [0,\|X_p\|] \to \B^n$ from $p$ to $q$ we have $\gamma'(0) = \|X_p\|^{-1} X_p$. By 
(\ref{eqn:dh}),
\begin{equation}\label{eqn:exph}
	\exp_0^h(X_0) = \frac{\tanh \|X_0\|}{\|X_0\|} X_0,\quad X_0\in T_0\B^n\cong \R^n.
\end{equation}

An affine hyperplane $H$ restricted to $\B^n$ can be viewed as an object of hyperbolic space or Euclidean 
space, depending on whether we choose the hyperbolic metric tensor $g^h$ or the Euclidean metric tensor $g^e$. 
The normal vector in any point $p\in H$ is also depends on the metric we choose and therefore we distinguish 
between the hyperbolic unit normal vector $N_p^h$ and the Euclidean unit normal vector $N_p^e$.  To more 
precise, $N_p^h\in T_p\B^n$ is a unit vector with respect to $g^h$ and $N_p^e$ is a unit vector with respect 
to $g^e$. The normal vectors are related and we include a proof of the following fact for the readers 
convenience.

\newpage
\begin{lemma}\label{lem:hnormal}
	Let $p\in \B^n$, $X_1,\ldots, X_{n-1}\in T_p\B^n$ be linearly independent and set 
$H=\mathrm{lin}(X_1,\ldots,X_{n-1})$. Then there are unique $N^h_p,N^e_p\in T_p\B^n$ such that:
	\begin{enumerate}
		\item[(a)] $N^h_p$ is orthonormal to $H$ with respect to $g^h$ and $N^e_p$ is orthonormal to $H$ with 
respect to $g^e$.
		\item[(b)] The frames $(X_1,\ldots,X_{n-1},N^h_p)$ and $(X_1,\ldots,X_{n-1},N^e_p)$ are positive 
oriented.
		\item[(c)] We have that
		\begin{equation}\label{eqn:normals}
			N^h_p = \sqrt{\frac{1-\|p\|^2}{1-(p\cdot N_p^e)^2}} (N_p^e-(p\cdot N_p^e)p).
		\end{equation}
	\end{enumerate}
\end{lemma}
\begin{proof}
	Let $g$ be a positive definite linear form on $T_p\B^n$. Then there is a uniquely determined vector $v\in 
T_p\B^n$, up to sign, such that $H$ is orthonormal to $v$ with respect to $g$, i.e., for $i=1,\ldots, n-1$, we 
have that $g(X_i,v)=0$ and $g(v,v)=1$. The sign of $v$ is determined by the condition that 
$(X_1,\ldots,X_{n-1},v)$ is a positive frame, which means that $\det(X_1,\ldots,X_{n-1},v)>0$.
	
	To conclude the proof, we only need to show $(c)$. We define $v$ as the vector obtained by the right-hand 
side of $(\ref{eqn:normals})$ and verify that $v$ satisfies $(a)$ and $(b)$ for $g^h$. Since $N^h_p$ is 
uniquely determined by these properties we conclude $N^h_p=v$.
\end{proof}
\bigskip

A \emph{hyperbolic ball} $B^n_h(p,r)$, is the set of all points $q\in \B^n$ with $d_h(q,p)\leq r$. For $p=0$ 
and by (\ref{eqn:dh}), we have
\begin{align}\label{eqn:hball}
	B^n_h(0,r) = \{q\in\B^n: d_h(q,0)\leq r\} = B^n_e(0,\tanh\, r).
\end{align}
Hence, the hyperbolic balls with center in the origin are also Euclidean balls in the projective model.
For a hyperbolic motion $m$ such that $m(0)=x$ we have $B^n_h(p,r) = m(B^n_h(0,r))=m(B^n_e(0,\tanh\, r))$. So 
hyperbolic balls in $\B^n$ are images of Euclidean balls with center $0$ under hyperbolic motions and 
therefore ellipsoids.

A subset $C\subseteq \B^n$ is hyperbolic convex if and only if $C$ is convex in the Euclidean sense as a 
subset of $\R^n$.
Planes and open, as well as closed, half-spaces are hyperbolic convex subsets.
A \emph{(hyperbolic) convex body} is a compact convex subset of $\B^n$.
Recall that $\K_0(\B^n)$ denotes set of convex bodies with non-empty interior contained in $\B^n$.

For a measurable subset $A\subset \B^n$, the \emph{hyperbolic volume} $\vol_n^h(A)$ in the projective model is
\begin{align}\label{eqn:volh}
	\vol_n^h(A) = \int\limits_{A} (1-\|x\|^2)^{-\frac{n+1}{2}}\, d\vol_n^e(x).
\end{align}
The \emph{hyperbolic Hausdorff distance} $\delta^h$ between hyperbolic convex bodies $K,L\in\K(\B^n)$ is 
defined by
\begin{align*}
	\delta^h(K,L) = \max \left\{\sup_{p\in K}\inf_{q\in L} d_h(p,q), \sup_{q\in L}\inf_{p\in K} 
d_h(p,q)\right\},
\end{align*}
and the \emph{hyperbolic volume difference metric} $\theta^{h}$ is
\begin{align*}
	\theta^{h}(K,L) = \vol_n^h(K\backslash L)+\vol_n^h(L\backslash K), \quad K,L\in \K_0(\B^n).
\end{align*}
Hyperbolic convex geometry can be viewed as study of notions on convex bodies $\K_0(\B^n)$ that are invariant 
under the group of motions $\M(\B^n)$. For instance, the hyperbolic volume $\vol_n^h$, the hyperbolic 
Hausdorff distance $\delta^h$ and the volume difference metric $\theta^{h}$, are all invariant with respect to 
hyperbolic motions and are therefore intrinsic notions of hyperbolic convex geometry.

\bigskip
The Euclidean support function $h_K(.)$ of a convex body $K$ measures for any direction $u\in\S^{n-1}$ the 
signed distance of a supporting hyperplane in direction $u$ to $K$ and the origin. Equivalently, one can use 
the orthogonal projection $K|\ell_u$ of $K$ to the line $\ell_u$ through the origin in direction $u$. Then 
\begin{align}\label{eqn:esupport}
	K|\ell_u = \mathrm{conv}(-h_K(-u)u, h_K(u)u).
\end{align}

We will define the hyperbolic support function in a similar way, but first we have to recall some further 
facts about hyperbolic space and the projective model.
For a closed convex subset $A\subset\H^n$ and a point $p\in \H^n$ there is a unique point $q\in A$ that 
minimizes the distance $d_h(p,q)$. The metric projection $p_A\colon\H^n\to A$ assigns to each point $p$ this 
unique point. Hence,
\begin{align*}
	d_h(p,p_A(p)) = \min_{q\in A} d_h(p,q).
\end{align*}
If $p\not\in A$, then $p_A(p)\in\bd\, A$ and the line spanned by $p_A(p)$ and $p$ is perpendicular to the 
boundary of $A$ in $p_A(p)$.
In particular, in the projective model $\B^n$ the projection $K|L$ of a convex body $K\in\K(\B^n)$ to a 
$k$-plane $L$ through $0$ is given by the Euclidean projection of $K$ to $L$. This follows, since for any 
point $p\in\B^n\cong T_p\B^n$ the normal directions $Y_p\in T_p\B^n$ are
determined by
\begin{align}\label{eqn:proj}
	0=g_p(p,Y_p) = \frac{p\cdot Y_p}{1-\|p\|^2} + \frac{\|p\|^2 (p\cdot Y_p)}{(1-\|p\|^2)^2} = \frac{p\cdot 
Y_p}{(1-\|p\|^2)^2}.
\end{align}

\begin{definition}[hyperbolic support function]\label{def:hsupport}
	Let $p\in \H^n$ be a fixed point and identify the set of unit vectors in $T_p\H^n$ with $\S^{n-1}$. For 
any hyperbolic convex body $K\subset \H^n$, the \emph{hyperbolic support function} $h^h_p(K,.)\colon 
\S^{n-1}\to \R$ of $K$ with respect to $p$ is defined by
	\begin{align}\label{eqn:defhsupport}
		K|\ell_{p,u} = \exp_p^h\left(\mathrm{conv}\left(-h^h_p(K,-u)u,h^h_p(K,u)u\right)\right),
	\end{align}
	where $\ell_{p,u}=\exp_p^h(\R u)$, i.e., the uniquely determined geodesic line in $p$ in direction $u$.
\end{definition}

In the projective model the hyperbolic support function for $p=0$ is related to the Euclidean support function 
in the following way.
\begin{lemma}\label{lem:hsupport}
	Let $K\subset \B^n$ be a convex body. For $u\in \S^{n-1}$, we have that
	\begin{align}\label{eqn:hsupport}
		\tanh\, h^h_0(K,u) = h_K(u).
	\end{align}
\end{lemma}
\begin{proof}
	Since $0\in\ell_{0,u}$ we have that the hyperbolic projection of $K$ to $\ell_{0,u}$ is the same as the 
Euclidean projection, see (\ref{eqn:proj}). Therefore, by (\ref{eqn:esupport}) and the definition of the 
hyperbolic support function, (\ref{eqn:defhsupport}), we have that
	\begin{align*}
		\mathrm{conv}\left(-h_K(-u)u,h_K(u)u\right) = 
\exp_0^h\left(\mathrm{conv}\left(-h_p^h(K,-u)u,h_p^h(K,u)u\right)\right).
	\end{align*}
	Using (\ref{eqn:exph}), we conclude (\ref{eqn:hsupport}).
\end{proof}

\subsection{Boundary structure of a convex body}

Let $K\in\K_0(\B^n)$. The boundary $\bd\, K$ is a hypersurface that is endowed with a Riemannian structure 
depending on the metric used in $\B^n$, i.e. either the Euclidean metric tensor $g^e$ or the hyperbolic metric 
tensor $g^h$.

The hyperbolic surface area element $d\vol_{\bd\, K}^h$ is related to the Euclidean surface area element 
$d\vol_{\bd\, K}^e$ in the following way:
The tangent space $T_x\bd K$ at a boundary point $x$ is a linear subspace of $T_x\B^n$ and by our 
identification of $T_x\B^n$ with $\R^n$ it does not depend on the underling metric tensor. By 
(\ref{eqn:normals}) and (\ref{eqn:volh}), we find that the Riemannian volume form induced by $g^h$ and $g^e$ 
on the boundary of $K$ are related, for $X_1,\ldots, X_{n-1}\in T_x \bd\, K$, by
\begin{align*}
	&d\vol_{\bd\, K}^h\left(X_1,\ldots,X_{n-1}\right) 
	= d\vol_n^h\left(X_1,\ldots,X_{n-1},N_x^h\right)\\
	&\quad = (1-\|x\|^2)^{-(n+1)/2} d\vol_n^e\left(X_1,\ldots,X_{n-1},(N_x^h\cdot N_x^e)N_x^e\right)\\
	&\quad = \sqrt{\frac{1-(x\cdot N^e_x)^2}{(1-\|x\|^2)^{n}}} d\vol_{\bd\, 
K}^e\left(X_1,\ldots,X_{n-1}\right).
\end{align*}
In particular, for $K\in\K_0(\B^n)$ and a measurable function $f\colon \bd\, K\to \R$, we have that
\begin{align}\label{eqn:volbd}
	\int\limits_{\bd\, K} f(x) \, d\vol_{\bd\, K}^h(x) = \int\limits_{\bd\, K} f(x) \sqrt{\frac{1-(x\cdot 
N^e_x)^2}{(1-\|x\|^2)^{n}}}\, d\vol_{\bd\, K}^e(x).
\end{align}

The Riemannian metric induced on the boundary of $K$ is denoted by $\hat{g}^h_p = g_p^h|_{\bd\, K}$ or 
$\hat{g}^e_p = g_p^e|_{\bd\, K}$. If $0\in\bd\, K$, then, by (\ref{eqn:tensor_equal}), $\hat{g}^h_0 = 
\hat{g}_0^e$. Therefore, in the projective model the hyperbolic curvature of $\bd\, K$ in $0$ is the same as 
the Euclidean curvature.
In the following theorem we collect the relations between the hyperbolic notions at a boundary point and the 
Euclidean ones in the projective model. This is definitely well-known and we again include a proof for 
convenience.

\begin{theorem}\label{thm:gauss}
	Let $M$ be a smooth orientable manifold of dimension $n-1$ immersed in $\B^n$. We denote the metric 
induced by $g^h$, resp.\ $g^e$, on $M$ by $\hat{g}^h$, resp.\ $\hat{g}^e$. The unique unit normal vector field 
along $M$ is denoted by $N^h$, resp.\ $N^e$. For $x\in M$, we have
	\begin{align*}
		\det\, \hat{g}^h_x = (1-\|x\|^2)^{-n}(1-(N_x^e\cdot x)^2).
	\end{align*}

	Denoting the covariant derivative on $\B^n$ by $\nabla^h$, resp.\ $\nabla^e$, the second fundamental form 
$\hat{h}^h$, resp.\ $\hat{h}^e$, is determined by 
	\begin{align*}
		\nabla^*_X Y &= \overline{\nabla}^*_X Y + \hat{h}^*(X,Y) N^*,
	\end{align*}
	where $\overline{\nabla}^h$, resp.\ $\overline{\nabla}^e$, denotes the induced covariant derivative on 
$M$.
	Then
	\begin{align*}
		\hat{h}^h_x = \hat{h}^e_x \left(1-\|x\|^2\right)^{-1/2}\left(1-(N_x^e\cdot x)^2\right)^{-1/2}.
	\end{align*}
	
	Let $S_x^h= (\hat{g}^h_x)^{-1}\hat{h}_x^h$ be the shape operator, i.e., the $(1,1)$-tensor equivalent to 
$\hat{h}^h_x$ and obtained by raising an index. For the Gauss--Kronecker curvature $H_{n-1}^*(M,x)=\det 
S_x^*$, we obtain
	\begin{align}\label{eqn:gaussstub}
		H_{n-1}^h(M,x) = H_{n-1}^e(M,x) \left(1-\|x\|^2\right)^{\frac{n+1}{2}}\left(1-(N_x^e\cdot 
x)^2\right)^{-\frac{n+1}{2}}.
	\end{align}
\end{theorem}
\begin{proof}
	Let $g^*$ be a Riemannian metric tensor of $\B^n$. We identify $T_xM$ with the $n-1$ dimensional subspace 
$\{X_x\in T_x\B^n\cong \R^n: X_x\cdot N_x^e=0\}$. For $X_x, Y_x\in T_xM$, the induced metric tensor 
$\hat{g}^*$ is determined by $\hat{g}^*_x(X_x,Y_x) = g^*_x(X_x,Y_x)$. In particular, for $g^h$ we have
	\begin{align*}
		\hat{g}^h_x(X_x,Y_x) = \frac{X_x\cdot Y_x}{1-\|x\|^2} + \frac{ (x\cdot X_x)(x\cdot Y_x) 
}{(1-\|x\|^2)^2}.
	\end{align*}
	We put $\overline{x} = x - (N_x^e\cdot x) N_x^e$. Then, for $X_x\in T_x M$, we have that $x\cdot X_x= 
\overline{x} \cdot X_x$ and $\|\overline{x}\| = \sqrt{\|x\|^2-(N_x^e\cdot x)^2}$.
	We define the matrix $A= (1-\|x\|^2)^{-1}\mathrm{Id}_n+(1-\|x\|^2)^{-2}\overline{x} \overline{x}^\top$ and 
obtain
	\begin{align*}
		A\frac{\overline{x}}{\|\overline{x}\|} = \frac{\overline{x}}{\|\overline{x}\|(1-\|x\|^2)} + 
\frac{\|\overline{x}\|^2 \overline{x}}{\|\overline{x}\|(1-\|x\|^2)^2} = \frac{1-(N_x^e\cdot 
x)^2}{(1-\|x\|^2)^2} \frac{\overline{x}}{\|\overline{x}\|}.
	\end{align*}
	For $v\in \overline{x}^\bot$, we have that $Av = (1-\|p\|^2)^{-1}v$.
	By definition of $A$, $\hat{g}^h_x(X_x,Y_x)=X_x^\top A Y_x$. We conclude that
	\begin{align*}
		\det \hat{g}^h_x = \det A = (1-\|x\|^2)^{-n}(1-(N_x^e\cdot x)^2).
	\end{align*}
	
	We know that $(\B^n,g^h)$ and $(\B^n,g^e)$ have the same (pre-)geodesics. This implies, see e.g.\ 
\cite{Eisenhart:1997}*{(40.7)}, that there is a function $\psi$ such that
	\begin{align*}
		\frac{\partial \log\det(g^h)}{\partial x^i} = \frac{\partial \log\det(g^e)}{\partial x^i} + 
2(n+1)\frac{\partial \psi}{\partial x^i}.
	\end{align*}
	Since $\log\det(g^h) = -(n+1)\log\left(1-\|x\|^2\right)$ and $\det(g^e)=1$, we conclude that
	$\frac{\partial \psi}{\partial x^i} = (1-\|x\|^2)^{-1}x_i$.
	Consequently, see e.g.\ \cite{Eisenhart:1997}*{(40.6)}, for the $1$-form $\rho$ defined by
	\begin{align}\label{eqn:rhorel}
		\rho(X_x) = (1-\|x\|^2)^{-1} (x\cdot X_x),
	\end{align}
	the covariant derivative with respect to $g^h$ can be written as
	\begin{align}\label{eqn:nablarel}
		\nabla^h_{X}Y = \nabla^e_X Y + \rho(X)Y+\rho(Y)X.
	\end{align}
	Combining $(\ref{eqn:normals})$, $(\ref{eqn:rhorel})$ and $(\ref{eqn:nablarel})$, a straightforward 
calculation shows that
	\begin{align*}
		g^h_x\left((\nabla^h_X Y)_x, N^h_x\right) = \left(1-\|x\|^2\right)^{-1/2}\left(1-(x\cdot 
N_x^e)^2\right)^{-1/2}\left((\nabla^e_X Y)_x\cdot N^e_x\right).
	\end{align*}
	This concludes the proof, since $\hat{h}^*_x(X_x,Y_x)=g^*_x((\nabla^*_X Y)_x,N^*_x)$ and 
(\ref{eqn:gaussstub}) follows from $\det S_x^h = \det(\hat{h}^h_x)/\det(\hat{g}_x^h)$.
\end{proof}

An immediate consequence of this theorem is, that for smooth convex bodies the hyperbolic Gauss--Kronecker 
curvature and the Euclidean Gauss--Kronecker curvature are related by (\ref{eqn:gaussstub}). This can be 
generalized to general convex bodies with the usual methods: For $K\in \K_0(\B^n)$ we call a boundary point 
$x\in\bd\, K$ \emph{normal}, if $\mathrm{bd}\, K$ at $x$ can locally be expressed as the graph of a convex 
function that is second order differentiable in $x$, see e.g.\ \cite{Hug:1996}*{p.\ 4}.
Hence, in a normal boundary point the Gauss--Kronecker curvature is defined and since almost all boundary 
points are normal, see e.g.\ \cite{Schneider:2014}*{Thm.\ 2.5.5}, we obtain a generalized notion of hyperbolic 
Gauss--Kronecker curvature $H_{n-1}^h(K,x)$ for arbitrary convex bodies $K\in\K_0(\B^n)$.

\begin{corollary}\label{cor:GaussKronecker}
	Let $K\in\K_0(\B^n)$. In a normal boundary point $x\in\bd\, K$, we have
	\begin{align}\label{eqn:gauss}
		H_{n-1}^h(K,x) = H_{n-1}^e(K,x) \left(\frac{1-\|x\|^2}{1-(N_x^e\cdot x)^2}\right)^{(n+1)/2}.
	\end{align}
\end{corollary}

\bigskip
The following proposition is well-known, see e.g. \cite{Schuett:1990}*{Lem.\ 3}. It is a change of variables 
formula, where we switch from integration in Cartesian coordinates to integration along rays from the origin 
with the directions parametrized by the boundary of a convex body.
\begin{proposition}[{Euclidean cone volume formula, see \cite{Schuett:1990}*{Lem.\ 3}}]\label{prop:econe}
	Let $K,L\in\K_0(\R^n)$ such that $L\subseteq K$ and $0\in\inter\, L$. For $x\in\bd\, K$ we set $\{x_L\} = 
\bd\, L \cap \mathrm{pos}\{x\}$.
	Then
	\begin{align*}
		\vol_n^e(K\backslash L) = \int\limits_{\bd\, K} \frac{h_K(N_x^e)}{\|x\|^n} 
\int\limits_{\|x_L\|}^{\|x\|} t^{n-1}\, dt\, d\vol_{bd\, K}^e(x).
	\end{align*}
\end{proposition}

There is an analog of the above in hyperbolic convex geometry.
\begin{proposition}[Hyperbolic cone volume formula]\label{lem:hyperconevol}
	Let $K,L\in \K_0(\B^n)$ such that $L\subseteq K$ and $0\in\inter\, L$. For $x\in\bd\, K$ we set $\{x_L\} = 
\bd\, L \cap \mathrm{pos}\{x\}$.
	Then
	\begin{align}\label{eqn:intform1}
		\vol_n^h(K\backslash L) = \int\limits_{\bd\, K} \frac{\sinh(h^h_0(K,N_x^e))}{\sinh(d_h(x,0))^n} 
\int\limits_{d_h(0,x_L)}^{d_h(0,x)} \sinh(t)^{n-1}\, dt\, d\vol_{\bd\, K}^h(x).
	\end{align}
\end{proposition}
\begin{proof}
	By Lemma \ref{lem:hsupport}, Proposition \ref{prop:econe} and (\ref{eqn:dh}),
	\begin{align*}
		\vol_n^h(K\backslash L) &= \int\limits_{\bd\, K} \frac{\tanh(h^h_0(K,N_x^e))}{\tanh(d_h(x,0))^n} 
\int\limits_{\tanh(d_h(x_L,0))}^{\tanh(d_h(x,0))} \frac{t^{n-1}}{(1-t^2)^{(n+1)/{2}}} \, dt\, d\vol_{\bd\, 
K}^e(x)\\
		&= \int\limits_{\bd\, K} \frac{\tanh(h^h_0(K,N_x^e))}{\tanh(d_h(x,0))^n} 
\int\limits_{d_h(x_L,0)}^{d_h(x,0)} \sinh(t)^{n-1}\,dt\, d\vol_{\bd\, K}^e(x).
	\end{align*}
	Since $h_K(N_x^e) = x\cdot N_x^e$ and by (\ref{eqn:dh}) and (\ref{eqn:volbd}), we have that
	\begin{align*}
		d\vol_{\bd\, K}^h(x)
		&= \frac{\cosh(d_h(x,0))^n}{\cosh(h^h_0(K,N_x^e))}\, d\vol_{\bd\, K}^e. \qedhere
	\end{align*}
\end{proof}

\subsection{A Euclidean model for real space forms}

Similar to the projective model, we may define a Euclidean model for space forms $\Sp^n(\lambda)$ of arbitrary 
curvature $\lambda$. 
Let
\begin{equation*}
	\B^n(\lambda) := \begin{cases}
	                 	\left(1/\sqrt{-\lambda}\right)\B^n & \text{ if $\lambda <0$,}\\
	                 	\R^n & \text{else}.
	                 \end{cases}
\end{equation*}
Further, define a Riemannian metric $g^\lambda$ on $\B^n(\lambda)$ by
\begin{equation}
	g^\lambda(X_p,Y_p) = \frac{X_p\cdot Y_p}{1+\lambda\|p\|^2} - \lambda\frac{(X_p\cdot p)(Y_p\cdot 
p)}{(1+\lambda\|p\|^2)^2}, \quad X_p,Y_p\in T_p\B^n(\lambda).
\end{equation}
Then $(\B^n(\lambda),g^\lambda)$ is a Riemannian manifold of constant sectional curvature $\lambda$.
By the Killing-Hopf Theorem there is, up to isometry, only one simply-connected and complete Riemannian 
manifold $\Sp^n(\lambda)$ of constant sectional curvature $\lambda\in\R$, see e.g.\ 
\cite{Kobayashi:1963}*{Ch.\ 6}, \cite{Lee:1997}*{Thm.\ 1.9} or \cite{ONeill:1983}*{Ch.\ 8, Cor.\ 25}.
Thus, for $\lambda \leq 0$, $(\B^n(\lambda), g^\lambda)$ is isometric to $\Sp^n(\lambda)$ and for $\lambda>0$, 
$(\B^n(\lambda),g^\lambda)$ is isometric to an open hemisphere of $\Sp^n(\lambda)$.

Euclidean straight lines intersected with $\B^n(\lambda)$ are geodesics in $(\B^n(\lambda),g^\lambda)$. 
Therefore the set of geodesically convex bodies in $(\B^n(\lambda),g^\lambda)$ is equivalent to 
$\K^n(\B^n(\lambda))$, i.e. the Euclidean convex bodies contained in $\B^n(\lambda)$. Note that in the 
spherical setting, $\lambda>0$, we define proper convex bodies as convex bodies contained in an open 
hemisphere. Hence, when investigating a fixed proper convex body $K\in\K(\Sp^n(\lambda))$, we may use the 
model $(\B^n(\lambda), g^\lambda)$ and identify $K$ with a convex body in $\K(\B^n(\lambda))$.

It is useful to define
\begin{align*}
	\tan^\lambda \alpha = \begin{cases}
	                        	\tanh\big(\sqrt{-\lambda}\alpha\big)/\sqrt{-\lambda} & \text{if $\lambda 
<0$},\\
	                        	\alpha & \text{if $\lambda = 0$},\\
	                        	\tan\big(\sqrt{\lambda}\alpha\big)/\sqrt{\lambda} & \text{if $\lambda >0$.}
	                        \end{cases}
\end{align*}
Then the geodesic distance $d_\lambda$ between a point $p\in\B^n(\lambda)$ and the origin is given by
\begin{align}\label{eqn:dlambda}
	\tan^\lambda d_\lambda(p,0) = d^e(p,0) = \|p\|.
\end{align}
For a geodesic ball $B^n_\lambda(0,\alpha)$ with center at the origin and geodesic radius $\alpha$ we have 
$B^n_\lambda(0,\alpha) = B^n_e(0,\tan^\lambda \alpha)$, i.e., geodesic balls with center at the origin are 
Euclidean balls.

The volume element in $(\B^n, g^\lambda)$ is 
\begin{equation}\label{eqn:lvol}
	d\vol_n^\lambda(p) = (1+\lambda\|p\|^2)^{-(n+1)/2}\, d\vol_n^e(p).
\end{equation}

For a convex body $K\subset \B^n(\lambda)$ we define a support function $h^\lambda_p(K,.)$ with respect to a 
fixed point $p\in\B^n(\lambda)$ similar to Definition \ref{def:hsupport}. If $p=0$, then 
\begin{align}\label{eqn:supportlambda}
	\tan^\lambda h^\lambda_0(K,u) = h_K(u), \quad u\in \S^{n-1}.
\end{align}

For a fixed convex body $K\in\K_0(\B^n(\lambda))$ and a regular boundary point $x\in\bd\, K$ we can compare 
the outer unit vector $N^\lambda_x$ with respect to $g^\lambda$ with the Euclidean outer unit normal. 
Analogous to Lemma \ref{lem:hnormal} we find that
\begin{equation}\label{eqn:lnormal}
	N_p^\lambda = \sqrt{\frac{1+\lambda\|x\|^2}{1+\lambda(x\cdot N_x^e)^2}} \left(N_x^e+\lambda(x\cdot 
N_x^e)x\right).
\end{equation}
This implies that
\begin{equation}\label{eqn:lboundary}
	d\vol_{\bd\, K}^\lambda(x) = \sqrt{\frac{1+\lambda(N_x^e\cdot x)^2}{(1+\lambda\|x\|^2)^{n}}} d\vol_{\bd\, 
K}^e(x).
\end{equation}
Finally, we can also adapt Theorem \ref{thm:gauss} and conclude that for normal boundary points $x\in\bd\, K$,
\begin{equation}\label{eqn:lgauss}
	H_{n-1}^\lambda(K,x) =  H_{n-1}^e(K,x) \left(\frac{1+\lambda\|x\|^2}{1+\lambda(N_x^e\cdot 
x)^2}\right)^{(n+1)/2}.
\end{equation}

For $\lambda=1$, we already obtained (\ref{eqn:lvol}), (\ref{eqn:supportlambda}), (\ref{eqn:lboundary}) and 
(\ref{eqn:lgauss}) in \cite{Besau:2015a}*{ (4.8), (4.3), (4.11) and (4.13)}.

\section{The Floating Body in Real Space Forms}

For a convex body $K\in\K_0(\Sp^n(\lambda))$ and $\delta>0$, we define the $\lambda$-floating body by
\begin{align}\label{def:lfloat}
	\F^\lambda_{\delta}\, K = \bigcap\left\{H^- : \vol_n^\lambda(K\cap H^+)\leq \delta^{\frac{n+1}{2}} 
\right\}.
\end{align}
In the Euclidean model $(\B^n(\lambda),g^\lambda)$, the $\lambda$-floating body is a weighted floating body \cite{Werner:2002}, 
that is, by (\ref{eqn:lvol}), we have $\F_\delta^\lambda\, K = \F_\delta^\mu\, K$ for $\mu=\vol_n^\lambda$.
Note that for $\lambda=0$, we obtain the well known Euclidean (convex) floating body $\F^e_\delta\, K$, see 
e.g.\ \cite{Schuett:1990}. For $\lambda=1$, we obtain the spherical floating body $\F^s_\delta\, K$ introduced 
in \cite{Besau:2015a}. Finally, for $\lambda = -1$ we obtain the new notion of hyperbolic floating body 
$\F^h_\delta\, K$.

By Proposition \ref{prop:euclpar} we have that
\begin{align}\label{eqn:lfloatpar}
	\F^\lambda_{\delta}\, K = \left[h_K-s_\delta^\lambda\right] = \bigcap_{v\in\S^{n-1}} 
H_{0,v,h_K(v)-s_\delta^\lambda(v)}^-,
\end{align}
where $s_\delta^\lambda$ is determined by
\begin{align*}
	\delta^{\frac{n+1}{2}} = \vol_n^\lambda\left( K \cap H^+_{0,v,h_K(v)-s_\delta^\lambda(v)} \right).
\end{align*}
	
The $\lambda$-floating body can be bounded by the Euclidean (convex) floating body in the following way.
\begin{lemma}\label{lem:lfloatbound}
	Let $K\in\K_0(\B^n(\lambda))$, $p\in\inter\, K$ and $0\leq \alpha < \beta$ be such that 
$B^n_\lambda(p,\alpha)\subset K \subseteq B^n_\lambda(p,\beta)$.
	We set
	\begin{align*}
		\delta_1 &:= \delta\left(1+\lambda \tan^\lambda\left(d_{\lambda}(0,p)-\alpha\right)^2\right),
		&
		\delta_2 &:= \delta\left(1+\lambda \tan^\lambda\left(d_{\lambda}(0,p)+\beta\right)^2\right).
	\end{align*}
	If $\delta>0$ is small enough so that $B^n_\lambda(p,\alpha)\subseteq \F^\lambda_\delta\, K$, then 
	\begin{align}\label{eqn:floatbound}
		\begin{cases}
			\F^e_{\delta_1}\, K \subseteq \F^{\lambda}_{\delta}\, K \subseteq \F^e_{\delta_2}\, K & \text{ if 
$\lambda <0$,}\\
			\F^e_{\delta_2}\, K \subseteq \F^{\lambda}_{\delta}\, K \subseteq \F^e_{\delta_1}\, K & \text{ if 
$\lambda >0$.}
		\end{cases}
	\end{align}
\end{lemma}
\begin{proof}
	It will be convenient to use the substitution $\tan^\lambda s(v,\delta) = h_K(v)-s_\delta^\lambda(v)$ in 
$(\ref{eqn:lfloatpar})$ to obtain
	\begin{align}\label{eqn:lfloatpar2}
		\F_\delta^\lambda \, K = \bigcap_{v\in\S^{n-1}} H^-_{0,v,\tan^\lambda s(v,\delta)},
	\end{align}
	where $s(v,\delta)$ is determined by
	\begin{align}\label{eqn:lfloatpar3}
		\delta^{\frac{n+1}{2}} = \vol_n^\lambda\left(K\cap H^+_{0,v,\tan^\lambda s(v,\delta)}\right).
	\end{align}
	Since $K \subseteq B^n_\lambda(p,\beta)$ and by (\ref{eqn:dlambda}), we conclude that, for all $q\in K$, 
$\| q\|\leq \tan^\lambda\left(d_{\lambda}(0,p)+\beta\right)$.
	This implies that
	\begin{align*}
		\begin{cases}
			(1+\lambda\|q\|^2)^{-1} \leq \cosh\left(\sqrt{-\lambda}(d_{\lambda}(0,p)+\beta)\right)^2 & 
\text{if $\lambda <0$,}\\
			(1+\lambda\|q\|^2)^{-1} \geq \cos\left(\sqrt{\lambda}(d_{\lambda}(0,p)+\beta)\right)^2 & \text{if 
$\lambda >0$.}
		\end{cases}
	\end{align*}
	Using this, $(\ref{eqn:dlambda})$ and $(\ref{eqn:lfloatpar3})$, we obtain
	\begin{align*}
		 \begin{cases}
		  	\delta^{\frac{n+1}{2}} \leq \cosh\left(\sqrt{-\lambda}(d_{\lambda}(0,p)+\beta)\right)^{n+1} 
\vol_n^e\left(K\cap H^+_{0,v,\tan^\lambda s(v,\delta)}\right) & \text{ if $\lambda <0$,}\\
		  	\delta^{\frac{n+1}{2}} \geq \cos\left(\sqrt{\lambda}(d_{\lambda}(0,p)+\beta)\right)^{n+1} 
\vol_n^e\left(K\cap H^+_{0,v,\tan^\lambda s(v,\delta)}\right) & \text{ if $\lambda >0$.}
		  \end{cases}
	\end{align*}
	For $\lambda <0$, let $t\left(v,\delta_2\right)$ be such that 
	\begin{align*}
		\delta_2^{\frac{n+1}{2}} = 
\left(\delta\cosh\left(\sqrt{-\lambda}(d_{\lambda}(0,p)+\beta)\right)^{-2}\right)^{\frac{n+1}{2}} = 
\vol^e_n\left(K\cap H^+_{0,v,t\left(v,\delta_2\right)}\right).
	\end{align*}
	Then $\tan^\lambda s(v,\delta) \leq t\left(v,\delta_2\right)$ and therefore
	\begin{align*}
		\F^\lambda_{\delta}\, K = [\tan^\lambda s(.,\delta)] \subseteq \left[t\left(.,\delta_2\right)\right] = 
\F^e_{\delta_2}\, K.
	\end{align*}
	For $\lambda >0$, an analogous argument gives $\tan^\lambda s(v,\delta) \geq t\left(v,\delta_2\right)$, 
which yields
	$\F^\lambda_\delta\, K\supseteq F^e_{\delta_2}\, K$.

	For the other inclusions we first note that $B^n_\lambda(p,\alpha)\subseteq \F_\delta^\lambda\, K 
\subseteq K$ implies
	$\|q\| \geq \left|\tan^\lambda\left(d_{\lambda}(0,p)-\alpha\right)\right|$, for all $q\in K\backslash 
\F_\delta^\lambda\, K$.
	By an argument analogous to the above, we find that
	\begin{align*}
			\begin{cases}
				\delta^{\frac{n+1}{2}}\geq \cosh\left(\sqrt{-\lambda}(d_{\lambda}(0,p)-\alpha)\right)^{n+1} 
\vol_n^e\left(K\cap H^+_{0,v,\tan^\lambda s(v,\delta)}\right) 
					&\text{ if $\lambda <0$,}\\
				\delta^{\frac{n+1}{2}}\leq \cos\left(\sqrt{\lambda}(d_{\lambda}(0,p)-\alpha)\right)^{n+1} 
\vol_n^e\left(K\cap H^+_{0,v,\tan^\lambda s(v,\delta)}\right) 
					&\text{ if $\lambda >0$.}
			\end{cases}
	\end{align*}
	Hence, we have that $\F^\lambda_{\delta} \, K\supseteq \F^e_{\delta_1}\, K$, for $\lambda <0$, 
respectively $\F^\lambda_{\delta} \, K \subseteq \F^e_{\delta_1}\, K$, for $\lambda >0$.
\end{proof}

A special case of Lemma \ref{lem:lfloatbound} for $\lambda=1$ has been obtained in \cite{Besau:2015a}*{Thm.\ 
5.2}.

Let $K\in\K_0(\B^n(\lambda))$ be such that $0\in \inter\, K$. For $x\in\bd\, K$ we denote by $x_\delta^K$ the 
uniquely determined intersection point of $\bd\, \F_\delta^\lambda\, K$ with the ray $\mathrm{pos}\{x\}$. We 
obtain the following corollary to Lemma \ref{lem:euclapprox}.

\begin{corollary}\label{cor:llocality}
	Let $K\in\K_0(\B^n(\lambda))$ be such that $0\in\inter\, K$ and let $x\in\bd\, K$ be a regular and exposed 
point. For $\varepsilon>0$ set $K'=K\cap B^n_\lambda(x,\varepsilon)$. Then $x\in\bd\, K'$ is a regular 
and exposed point of $K'$. Moreover, there exists $\delta_\varepsilon$ such that for all 
$\delta<\delta_\varepsilon$, we have that $x_\delta^{K'} = x_\delta^K$.
\end{corollary}
\begin{proof}
	We may move $K$ by an isometry of $(\B^n(\lambda),g^\lambda)$ so that $0\in\inter\, K'$.
	Since the geodesic balls $B^n_\lambda(p,\alpha)$ are ellipsoids, there exists a small Euclidean 
ball with the same center $p$ that is contained in $B^n_\lambda(p,\alpha)$. Hence, without loss of generality, 
there is $\eta:=\eta(\varepsilon,x,K)>0$ such that $0\in \inter(K\cap B^n_e(x,\eta)) \subseteq \inter\, K'$. 
We set $K''=K\cap B^n_e(x,\eta)$.
	
	We apply Lemma \ref{lem:euclapprox} for $\mu=\vol_n^\lambda$ and $\varepsilon=\eta$, and obtain 
$(\F_\delta^\lambda\, K)\cap\mathrm{pos}\{x\} = (\F_\delta^\lambda\, K'') \cap \mathrm{pos}\{x\}$, for all 
$\delta<\delta_\eta$. Note that $K\subseteq L$ implies $F_\delta^\mu\, K\subseteq F_\delta^\mu\, L$. This 
yields
	\begin{align*}
		(\F_\delta^\lambda\, K)\cap\mathrm{pos}\{x\} &= (\F_\delta^\lambda\, K'') \cap \mathrm{pos}\{x\} \\ 
		&\subseteq (\F_\delta^\lambda\, K')\cap \mathrm{pos}\{x\} \subseteq (\F_\delta^\lambda\, K) \cap 
\mathrm{pos}\{x\}.
	\end{align*}
	Hence, $(\F_\delta^\lambda\, K)\cap \mathrm{pos}\{x\} = (\F_\delta^\lambda\, K')\cap \mathrm{pos}\{x\}$ 
and therefore $x^K_\delta = x^{K'}_\delta$ for all $\delta< \delta_{\eta} =:\delta_{\varepsilon}$.
\end{proof}

\subsection{Proof of Theorem \ref{thm:intro_main}}
We are now ready to prove Theorem \ref{thm:intro_main}. For a (proper) convex body $K\in\K_0(\Sp^n(\lambda))$ 
we consider the Euclidean model $(\B^n(\lambda),g^\lambda)$ for $\Sp^n(\lambda)$ and identify $K$ with an 
Euclidean convex body in $\B^n(\lambda)$ such that $0\in\inter\, K$.

Analogous to Proposition \ref{lem:hyperconevol} we obtain the following.
\begin{proposition}\label{prop:lconevol}
	Let $K,L\in\K_0(\B^n(\lambda))$ be such that $L\subseteq K$ and $0\in\inter\, L$. For $x\in\bd\, K$ we set 
$\{x_L\}=\bd\, L \cap \mathrm{pos}\{x\}$. Then
	\begin{align*}
		\vol_n^\lambda\left(K\backslash L\right) = \int\limits_{\bd\, K} \frac{x\cdot N_x^e}{\|x\|^n} 
\int\limits_{\|x_L\|}^{\|x\|} \frac{t^{n-1}}{(1+\lambda t^2)^{\frac{n+1}{2}}}\, dt\, d\vol_{\bd\, K}^e(x).
	\end{align*}
\end{proposition}

Let $\delta>0$ be small enough, so that $0\in\inter\, \F_\delta^\lambda\, K$. 
To prove Theorem \ref{thm:intro_main} we have to show that
\begin{align*}
	\lim_{\delta\to 0^+} \frac{\vol_n^\lambda\left(K\backslash \F_\delta^\lambda\, K\right)}{\delta} = 
c_n\int\limits_{\bd\, K} H^\lambda_{n-1}(K,x)^{\frac{1}{n+1}}\, d\vol_{\bd\, K}^\lambda(x).
\end{align*}
By Proposition \ref{prop:lconevol}, we have
\begin{align}\label{eqn:maineq1}
	\frac{\vol_n^\lambda\left(K\backslash \F_\delta^\lambda\, K\right)}{\delta} = \int\limits_{\bd\, K} 
\frac{x\cdot N_x^e}{\delta\|x\|^n} \int\limits_{\|x_{\delta}^\lambda\|}^{\|x\|} \frac{t^{n-1}}{(1+\lambda 
t^2)^{\frac{n+1}{2}}}\, dt\, d\vol_{bd\, K}^e(x).
\end{align}

We will first show that the integrand is uniformly bounded in $\delta$ by an integrable function.
\begin{lemma}\label{lem:proof1}
	Let $K\in \K_0(\B^n(\lambda))$ and $0\in\inter\, K$. Then there exists $\alpha,\beta>0$ and $\delta_0>0$ 
such that $B^n_\lambda(0,\alpha)\subseteq \inter\, \F_\delta^\lambda\, K$ for all $\delta\leq \delta_0$ and $K 
\subset B^n_\lambda(0,\beta)$.
	Furthermore, for regular boundary points $x\in\bd\, K$ and for $0<\delta <\delta_0$, define
	\begin{align}\label{eqn:main2}
		f(x,\delta) := \frac{x\cdot N_x^e}{\delta\|x\|^n} \int\limits_{\|x_\delta^\lambda\|}^{\|x\|} 
\frac{t^{n-1}}{(1+\lambda t^2)^{\frac{n+1}{2}}}\, dt.
	\end{align}
	Then $f(x,\delta)$ is bounded from above for all $\delta<\delta_0$ by an integrable function $g(x)$, for 
almost all $x\in\bd\, K$.
\end{lemma}
\begin{proof}
	Since $0\in\inter\, K$, there is $\delta_0>0$ such that $0\in\inter\, \F_{\delta_0}^\lambda\, K$. Thus 
there exists $\alpha>0$ such that $B^n_\lambda(0,\alpha)\subseteq \inter \F_{\delta_0}^\lambda\, K$ and, by 
monotonicity, this yields $B^n_\lambda(0,\alpha)\subseteq \F_{\delta}^\lambda\, K$, for all $\delta\leq 
\delta_0$. Furthermore, since $K$ is bounded, there exists $\beta>0$ such that $K\subseteq 
B^n_\lambda(0,\beta)$. By (\ref{eqn:dlambda}), this implies that $\tan^\lambda \alpha \leq 
\|x_\delta^\lambda\| \leq \|x\| \leq \tan^\lambda \beta$ for all $\delta<\delta_0$. 
	
	We set
	\begin{align}\label{eqn:tdelta}
		\widetilde{\delta} := \begin{cases}
		                      	\delta\cosh(\sqrt{-\lambda}\alpha)^{-2} & \text{if $\lambda <0 $,}\\
		                      	\delta\cos(\sqrt{\lambda}\beta)^{-2} & \text{if $\lambda > 0$.}
		                      \end{cases}
	\end{align}
	By Lemma \ref{lem:lfloatbound}, we have that $\F_{\widetilde{\delta}}^e\, K\subseteq \F_\delta^\lambda\, 
K$ and therefore
	$\|x_{\widetilde{\delta}}^e-x\| \geq \|x_\delta^\lambda - x\|$.
	For $\|x_\delta^\lambda\|\leq t\leq \|x\|$, we obtain
	\begin{align}\label{eqn:bound2}
		\frac{1}{(1+\lambda t^2)^{\frac{n+1}{2}}} \leq 
			\begin{cases}
		    	\cosh(\sqrt{-\lambda}\beta)^{n+1} & \text{if $\lambda <0$,}\\
		    	\cos(\sqrt{\lambda}\alpha)^{n+1} & \text{if $\lambda >0$.}
		    \end{cases}
	\end{align}
	We conclude that
	\begin{align}\label{eqn:bound3}
		\frac{1}{\delta}\left(\frac{x}{\|x\|}\cdot N_x^e\right) \int\limits_{\|x_\delta^\lambda\|}^{\|x\|} 
\left(\frac{t}{\|x\|}\right)^{n-1}(1+\lambda t^2)^{-\frac{n+1}{2}}\, dt \leq C \left(\frac{x}{\|x\|}\cdot 
N_x^e\right) \frac{\left\|x_{\widetilde{\delta}}^e-x\right\|}{\widetilde{\delta}},
	\end{align}
	where we put, for $\lambda <0$, $C:=\cosh(\sqrt{-\lambda}\beta)^{n+1}\cosh(\sqrt{-\lambda}\alpha)^{-2}$, 
respectively, for $\lambda >0$, $C:=\cos(\sqrt{\lambda}\alpha)^{n+1}\cos(\sqrt{\lambda}\beta)^{-2}$.
	
	This concludes the proof, since the right-hand side of $(\ref{eqn:bound3})$ is the same integrand we 
obtain for the Euclidean (convex) floating body and is therefore bounded uniformly in $\widetilde{\delta}$ by 
an integrable function for almost all $x\in\bd\, K$, by \cite{Schuett:1990}*{Lem.\ 5 and Lem.\ 6}.
\end{proof}

It only remains to show that (\ref{eqn:main2}) converges point-wise for almost all boundary points. Since 
almost all boundary points are normal, see page \pageref{cor:GaussKronecker}, it is sufficient to show the 
following.
\begin{lemma}\label{lem:proof2}
	Let $K\in \K_{0}(\B^n(\lambda))$ and $0\in\inter\, K$. Then, for normal boundary points $x\in\bd\, K$, we 
have that
	\begin{align}\label{eqn:gausseucld}
		\lim_{\delta\to 0^+}
		 \frac{x\cdot N_x^e}{\delta\|x\|^n} \int\limits_{\|x_\delta^\lambda\|}^{\|x\|} 
\frac{t^{n-1}}{(1+\lambda t^2)^{\frac{n+1}{2}}}\, dt 
		 = c_n \frac{H^e_{n-1}(K,x)^{\frac{1}{n+1}}}{(1+\lambda\|x\|^2)^{\frac{n-1}{2}}},
	\end{align}
	where $c_n = \tfrac{1}{2}\left((n+1)/\kappa_{n-1}\right)^{2/(n+1)}$.
\end{lemma}
\begin{proof}
	A normal boundary point $x\in\bd\, K$ has a unique outer unit normal $N_x^e$ and the 
Gauss--Kronecker curvature $H_{n-1}^e(K,x)$ exists.
	We first consider the case that $H^e_{n-1}(K,x)=0$ and show that the left-hand side of 
(\ref{eqn:gausseucld}) converges to $0$, for $\delta\to 0^+$. 
	With $\widetilde{\delta}$ as defined by (\ref{eqn:tdelta}) in Lemma \ref{lem:lfloatbound}, we find again 
the upper bound (\ref{eqn:bound3}). 
	The function in the upper bound is the same as the integrand we obtain for the Euclidean convex floating 
body and therefore it converges to $0$, for $\widetilde{\delta}\to 0^+$, by \cite{Schuett:1990}*{Lem.\ 
7 and Lem.\ 10}. This implies that
	\begin{align*}
		\limsup_{\delta\to 0^+} \frac{x\cdot N_x^e}{\delta\|x\|^n} \int\limits_{\|x_\delta^\lambda\|}^{\|x\|} 
\frac{t^{n-1}}{(1+\lambda t^2)^{\frac{n+1}{2}}}\, dt 
		\leq C \limsup_{\widetilde{\delta}\to 0^+} \left(\frac{x}{\|x\|}\cdot N_x^e\right) 
\frac{\left\|x_{\widetilde{\delta}}^e-x\right\|}{\widetilde{\delta}} 
		= 0.
	\end{align*}
	
	Next, let $H^e_{n-1}(K,x)>0$.
	Let $\varepsilon>0$ be arbitrary and set $K' = K\cap B^n_\lambda(x,\varepsilon)$.
	Furthermore, let $p$ be a point inside $K'$ and on the segment spanned by $x$ and the origin, that is, 
$p\in\inter\, K' \cap \mathrm{pos}\{x\}$.
	For $\alpha=0$ and $\beta=\varepsilon$ define $\delta_1$ and $\delta_2$ as in Lemma \ref{lem:lfloatbound}.
	Then, for $\delta$ small enough, we have $p\in \inter\, \F_\delta^\lambda\, K'$. Thus, for $\lambda<0$,
	$\F^e_{\delta_1}\, K'\subseteq \F_\delta^\lambda\, K' \subseteq \F^e_{\delta_2}\, K'$, respectively, for 
$\lambda >0$, $\F^e_{\delta_1}\, K'\supseteq \F_\delta^\lambda\, K' \supseteq \F^e_{\delta_2}\, K'$. Corollary 
\ref{cor:llocality} implies that
	\begin{align*}
		\{x_\delta^\lambda\}=\bd\, \F_\delta^\lambda\, K \cap \mathrm{conv}(x,p) = \bd\, \F_\delta^\lambda\, 
K'\cap\mathrm{conv}(x,p).
	\end{align*}
	This yields
	\begin{align}\label{eqn:bound1}
		\begin{cases}
			\|x-x_{\delta_1}^e\| \geq \|x-x_{\delta}^\lambda\| \geq \|x-x_{\delta_2}^e\| & \text{if $\lambda 
<0$,}\\
			\|x-x_{\delta_1}^e\| \leq \|x-x_{\delta}^\lambda\| \leq \|x-x_{\delta_2}^e\| & \text{if 
$\lambda>0$.}
		\end{cases}
	\end{align}
	Hence for $\lambda <0$, we have
	\begin{align*}
		\frac{x\cdot N_x^e}{\delta\|x\|^n} \int\limits_{\|x_\delta^\lambda\|}^{\|x\|} 
\frac{t^{n-1}}{(1+\lambda t^2)^{\frac{n+1}{2}}}\, dt
		& \geq 
		\frac{\frac{x}{\|x\|}\cdot N_x^e}{(1+\lambda\|x_\delta^\lambda\|^2)^{\frac{n+1}{2}}} 
\left(\frac{\|x_\delta^\lambda\|}{\|x\|}\right)^{n-1} \frac{\delta_2}{\delta} 
\frac{\|x-x_{\delta_2}^e\|}{\delta_2},\\
		\frac{x\cdot N_x^e}{\delta\|x\|^n} \int\limits_{\|x_\delta^\lambda\|}^{\|x\|} 
\frac{t^{n-1}}{(1+\lambda t^2)^{\frac{n+1}{2}}}\, dt
		&\leq
		\frac{\frac{x}{\|x\|}\cdot N_x^e}{(1+\lambda \|x\|^2)^{\frac{n+1}{2}}} \frac{\delta_1}{\delta} 
\frac{\|x-x_{\delta_1}^e\|}{\delta_1}.
	\end{align*}
	Conversely, if $\lambda>0$, then
	\begin{align*}
		\frac{x\cdot N_x^e}{\delta\|x\|^n} \int\limits_{\|x_\delta^\lambda\|}^{\|x\|} 
\frac{t^{n-1}}{(1+\lambda t^2)^{\frac{n+1}{2}}}\, dt
		& \geq 
		\frac{\frac{x}{\|x\|}\cdot N_x^e}{(1+\lambda\|x\|^2)^{\frac{n+1}{2}}} 
\left(\frac{\|x_\delta^\lambda\|}{\|x\|}\right)^{n-1} \frac{\delta_1}{\delta} 
\frac{\|x-x_{\delta_1}^e\|}{\delta_1},\\
		\frac{x\cdot N_x^e}{\delta\|x\|^n} \int\limits_{\|x_\delta^\lambda\|}^{\|x\|} 
\frac{t^{n-1}}{(1+\lambda t^2)^{\frac{n+1}{2}}}\, dt
		&\leq
		\frac{\frac{x}{\|x\|}\cdot N_x^e}{(1+\lambda \|x_\delta^\lambda\|^2)^{\frac{n+1}{2}}} 
\frac{\delta_2}{\delta} \frac{\|x-x_{\delta_2}^e\|}{\delta_2}.
	\end{align*}
	
	To finish the proof we first notice that the functions that appear on the right-hand side of the above 
inequalities are again related to the integrand that is obtained for the Euclidean convex floating body. 
Hence, by \cite{Schuett:1990}*{Lem.\ 7 and Lem.\ 11}, for $\delta_*\in\{\delta_1,\delta_2\}$,
	\begin{align*}
		\lim_{\delta_*\to 0^+} \frac{x}{\|x\|}\cdot N_x^e \frac{\|x-x_{\delta_*}^e\|}{\delta_*} = c_n 
H_{n-1}^e(K,x)^{\frac{1}{n+1}}.
	\end{align*}
	By the choice of $p$, we have $\left|\|x\|-\varepsilon\right| \leq \|p\| \leq \|x\|$. For $\lambda <0$, by 
the definition of $\delta_1$ and $\delta_2$, there exist positive constants $C_1, C_2>0$, such that
	\begin{align*}
		\frac{\delta_1}{\delta} \leq (1+\lambda\|x\|^2)(1+C_1\varepsilon) \text{ and } \frac{\delta_2}{\delta} 
\geq (1+\lambda\|x\|^2)(1-C_2\varepsilon).
	\end{align*}
	Therefore
	\begin{align*}
		\limsup_{\delta\to 0^+} \frac{x\cdot N_x^e}{\delta\|x\|^n} \int\limits_{\|x_\delta^\lambda\|}^{\|x\|} 
\frac{t^{n-1}}{(1+\lambda t^2)^{\frac{n+1}{2}}}\, dt
		&\leq \limsup_{\delta\to 0^+}
		\frac{\frac{x}{\|x\|}\cdot N_x^e}{(1+\lambda \|x\|^2)^{\frac{n+1}{2}}} \frac{\delta_1}{\delta} 
\frac{\|x-x_{\delta_1}^e\|}{\delta_1}\\
		&\leq c_n \frac{H_{n-1}^e(K,x)^{\frac{1}{n+1}}}{(1+\lambda\|x\|^2)^{\frac{n-1}{2}}} 
(1+C_1\varepsilon),
	\end{align*}
	and similarly
	\begin{align*}
		\liminf_{\delta\to 0^+} \frac{x\cdot N_x^e}{\delta\|x\|^n} \int\limits_{\|x_\delta^\lambda\|}^{\|x\|} 
\frac{t^{n-1}}{(1+\lambda t^2)^{\frac{n+1}{2}}}\, dt
		&\geq c_n \frac{H_{n-1}^e(K,x)^{\frac{1}{n+1}}}{(1+\lambda\|x\|^2)^{\frac{n-1}{2}}} 
(1-C_2\varepsilon).
	\end{align*}
	Since $\varepsilon>0$ was arbitrary, we conclude (\ref{eqn:gausseucld}), for $\lambda<0$. For $\lambda >0$ 
the argument is analogous.	
\end{proof}

Combining Lemma \ref{lem:proof1}, Lemma \ref{lem:proof2}, (\ref{eqn:lboundary}) and (\ref{eqn:lgauss}), we 
conclude
\begin{align}\label{eqn:proof_main}
	\lim_{\delta\to 0^+} \frac{\vol_n^\lambda\left(K\backslash \F_\delta^\lambda\, K\right)}{\delta} 
	&= c_n\int\limits_{\bd\, K} \frac{H_{n-1}^e(K,x)^{\frac{1}{n+1}}}{(1+\lambda\|x\|^2)^{\frac{n-1}{2}}} \, 
d\vol_{\bd\, K}^e(x) \\
	&= c_n\int\limits_{\bd\, K} H_{n-1}^\lambda(K,x)^{\frac{1}{n+1}} \, d\vol_{\bd\, K}^\lambda(x).\notag
\end{align}
This finishes the proof of Theorem \ref{thm:intro_main}.

\section{The Floating Area in Real Space Forms}

We denote the Borel $\sigma$-algebra of a metric space $(X,d)$ by $\mathcal{B}(X)$. 
For $K\in \K_0(\B^n(\lambda))$ and $\omega\in\mathcal{B}(\B^n(\lambda))$ we conclude, by Theorem 
\ref{thm:intro_main2} and (\ref{eqn:proof_main}), that
\begin{align}\label{eqn:lfloatmeasure}
	\lim_{\delta\to 0^+} \frac{\vol_n^\lambda\left((K\backslash \F_\delta^\lambda\, 
K)\cap\omega\right)}{\delta} = \int\limits_{(\bd\, K)\,\cap\, \omega} H_{n-1}^\lambda(K,x)^{\frac{1}{n+1}}\, 
d\vol_{bd\, K}^\lambda(x).
\end{align}

\begin{definition}
	The \emph{$\lambda$-floating measure} $\Omega^\lambda(.,\!.)$ is defined, for 
$K\!\!\in\!\K_0(\Sp^n(\lambda))$ and $\omega\in\B(\Sp^n(\lambda))$, by
	\begin{align}\label{eqn:deffloatarea}
		\Omega^\lambda(K,\omega) = \int\limits_{(\bd\, K)\,\cap\, \omega} 
H_{n-1}^\lambda(K,x)^{\frac{1}{n+1}}\, d\vol_{bd\, K}^\lambda(x).
	\end{align}
	The $\lambda$-floating area $\Omega^\lambda(.)$ of a convex body $K$ is 
$\Omega^\lambda(K)=\Omega^\lambda(K,\Sp^n(\lambda))$.
\end{definition}

For $\lambda>0$, we distinguish between proper and non-proper convex bodies. Recall that a convex body is 
proper, if and only if it does not contain two antipodal points. Equivalently, a convex body is 
proper if and only if it is contained in an open half-space (open hemisphere). 
By (\ref{eqn:lfloatmeasure}), the definition (\ref{eqn:deffloatarea}) makes sense for proper convex bodies.
Non-proper convex bodies $K$ with non-empty interior are either the whole space or a \emph{lune}. 
A $k$-lune is the convex hull $\mathrm{conv}(S,L)$ of a $k$-dimensional totally geodesic subspace ($k$-sphere) 
$S$ and a proper convex body $L$ in an $(n-k-1)$-dimensional totally geodesic subspace polar to $S$.
Thus, for non-proper convex bodies we either have $K=\Sp^n(\lambda)$ and therefore $\bd\, K=\emptyset$ or $K$ 
is a lune and the boundary is ``flat'', that is, $H_{n-1}^\lambda(K,x) = 0$ for almost all boundary points 
$x\in\bd\, K$. Therefore we set $\Omega^\lambda(K,\omega)=0$ for non-proper convex bodies. See also
\cite{Besau:2015a} for more details.

Finally, from the definition (\ref{eqn:deffloatarea}) it is obvious that the $\lambda$-floating area vanishes 
for ``flat'' bodies. In particular, the $\lambda$-floating area for polytopes is zero.

\subsection{Proof of Theorem \ref{thm:intro_main2}}

We first prove the valuation property. The proof is analogously to the proof for the affine surface area in 
\cite{Schuett:1993}. Let $K, L \in \mathcal{K}_0(\Sp^n(\lambda))$ such that $K\cup 
L\in\mathcal{K}_0(\Sp^n(\lambda))$. We have to show
\begin{align}\label{eqn:propval}
	\Omega^\lambda(K,\omega)+\Omega^\lambda(L,\omega) = \Omega^\lambda(K\cup L, \omega) + \Omega^\lambda(K\cap 
L,\omega).
\end{align}
We first observe 
\begin{align*}
		\bd\, K &= (\bd\, K \cap \bd\, L) \cup (\bd\, K \cap \inter\, L) \cup (\bd\,K \cap L^c),\\
		\bd\, L &= (\bd\, K \cap \bd\, L) \cup (\inter\, K\cap \bd\, L) \cup (K^c\cap \bd\, L),\\
		\bd(K\cap L) &= (\bd\, K\cap\bd\, L) \cup (\bd\, K\cap \inter\, L) \cup (\inter\, K \cap \bd\, L),\\
		\bd(K\cup L) &= (\bd\, K\cap\bd\, L) \cup (\bd\, K \cap L^c) \cup (K^c\cap \bd\, L),
\end{align*}
where $K^c = \Sp^n(\lambda)\backslash K$ and $L^c=\Sp^n(\lambda)\backslash L$.
Then (\ref{eqn:propval}) reduces to
\begin{align}\label{eqn:propval2}\begin{split}
	&\int\limits_{(\bd\, K\,\cap\, \bd\, L)\,\cap\, \omega} H_{n-1}^\lambda(K,x)^{\frac{1}{n+1}} +
	H_{n-1}^\lambda(L,x)^{\frac{1}{n+1}}\, d\vol_{bd\, L}^\lambda(x)\\
	&=
	\int\limits_{(\bd\, K\,\cap\, \bd\, L)\,\cap \,\omega} H_{n-1}^\lambda(K\cup L,x)^{\frac{1}{n+1}} +
	H_{n-1}^\lambda(K\cap L,x)^{\frac{1}{n+1}}\, d\vol_{bd\, L}^\lambda(x).
	\end{split}
\end{align}
Locally around any point $x\in\bd K \cap \bd\, L$, we use the Euclidean model $(\B^n(\lambda),g^\lambda)$. 
Hence, $H_{n-1}^\lambda(K,x)$ and $H_{n-1}^\lambda(K,x)$ are related by (\ref{eqn:lgauss}) at normal boundary 
points $x\in\bd\, K$. With \cite{Schuett:1993}*{Lem.\ 5}, we conclude that 
\begin{align*}
	H_{n-1}^\lambda(K\cup L,x) &= \min\left\{ H_{n-1}^\lambda(K,x), H_{n-1}^\lambda(L,x)\right\},\\
	H_{n-1}^\lambda(K\cap L,x) &= \max\left\{ H_{n-1}^\lambda(K,x), H_{n-1}^\lambda(L,x)\right\}.
\end{align*}
This verifies (\ref{eqn:propval2}) and therefore $\Omega^\lambda(.,\omega)$ is a valuation on 
$\K_0(\Sp^n(\lambda))$.

Since $\Omega^\lambda(.,\omega)$ can be seen as a curvature measure on $\bd\, K$, the proof of the 
upper-semicontinuity of $\Omega^\lambda(.,\omega)$ is analogous to the proofs presented in \cite{Ludwig:2001}. 
We include the following short argument:
Let $(K_\ell)_{\ell\in\mathbb{N}}$ be a sequence of convex bodies converging to $K\in\K_0(\Sp^n(\lambda))$. By 
the valuation property we may assume, for $\lambda>0$, that $K\cup \bigcup_{\ell\in\mathbb{N}} K_\ell$ is 
contained in an open half-space. We choose a Euclidean model $(\B^n(\lambda),g^\lambda)$ and identify $K_\ell$ 
and $K$ with Euclidean convex bodies. Hence,
\begin{align*}
	\Omega^\lambda(K,\omega) = \int\limits_{(\bd\, K)\,\cap\, \omega} 
\frac{H_{n-1}^e(K,x)^{\frac{1}{n+1}}}{(1+\lambda\|x\|^2)^{\frac{n-1}{2}}} \, d\vol_{\bd\, K}^e(x).
\end{align*}
The density $f^\lambda(x) := (1+\lambda\|x\|^2)^{-(n-1)/2}$ is continuous and
\begin{align*}
	\Omega^0(K,\omega) = \int\limits_{(\bd\, K)\,\cap\, \omega} H_{n-1}^e(K,x)^\frac{1}{n+1}\, d\vol_{\bd\, 
K}^e(x)
\end{align*}
is the classical affine surface area. Thus $\Omega^0(.,\omega)$ is upper semicontiuous, see e.g., 
\cite{Lutwak:1991}. To finish the prove let $\varepsilon>0$. By compactness of $K$ and continuity of 
$f^\lambda$, we find a finite partition of $\bd\,K\cap \omega$ into measurable subsets $(\omega_j)_{j=0}^N$ 
and points $x_j\in \omega_j$ such that $| f^\lambda(x)-f^\lambda(x_j)| < \varepsilon$, for all $x\in 
\omega_j$. Therefore
\begin{align*}
	\limsup_{\ell\in\mathbb{N}} \Omega^\lambda(K_\ell,\omega)
	&\leq \sum_{j=0}^N (f^\lambda(x_j)+\varepsilon) \limsup_{\ell\in\mathbb{N}} \Omega^0(K_\ell,\omega_j)\\
	&= \sum_{j=0}^N (f^\lambda(x_j)+\varepsilon) \Omega^0(K,\omega_j) \leq
	\Omega^\lambda(K,\omega) +\varepsilon 2N\Omega^0(K).
\end{align*}
Since $\varepsilon>0$ was arbitrary, this proves the upper semicontinuity of $\Omega^\lambda(.,\omega)$.

Finally, the fact that $\Omega^\lambda(.,\omega)$ is invariant under isometries is obvious, since it is a 
intrinsic notion. 
For $\lambda=0$, the \emph{equi-affine} transformations are characterized as bijective automorphisms that map 
lines to lines, are measurable and preserve volume.
Note that $\Sp^n(\lambda)$ is \emph{rigid} for $\lambda\neq 0$, in the sense that there are no bijective 
mappings $\varphi\colon\Sp^n(\lambda)\to\Sp^n(\lambda)$, other than isometries, that map geodesics to 
geodesics, are measurable and preserve volume.

\subsection{Isoperimetric inequality}

For $K\in\K_0(\R^n)$, the classical and well-known inequality associated with the affine surface area is
\begin{align}\label{eqn:affinineq}
	\mathrm{as}_1(K) \leq n \kappa_n^{\frac{2}{n+1}} \vol_n^e(K)^{\frac{n-1}{n+1}},
\end{align}
with equality if and only if $K$ is an ellipsoid. A natural question is, whether an extension of this 
inequality holds for the $\lambda$-floating area. Inequality (\ref{eqn:affinineq}) can be restated as: For 
all convex bodies of volume $\alpha$ the ball of radius $(\alpha/\kappa_n)^{1/n}$ maximizes the affine surface 
area, i.e.,
\begin{align*}
	 \sup_{K\in \K_0(\R^n)} \left\{ \mathrm{as}_1(K) : \vol_n^e(K)=\alpha\right\} = 
\mathrm{as}_1\left(B^n_e\left(0,\left(\alpha/\kappa_n\right)^{1/n}\right)\right).
\end{align*}
Therefore, we define
\begin{align}\label{eqn:liso}
	C^\lambda(\alpha) := \sup_{K\in \K_0(\Sp^n(\lambda))}\left\{ \Omega^\lambda(K): \vol_n^\lambda(K)=\alpha 
\right\}.
\end{align}
Then, for $\lambda=0$ and by (\ref{eqn:affinineq}), we conclude
\begin{align*}
	C^0(\alpha) = n\kappa_n^{\frac{2}{n+1}} \alpha^{\frac{n-1}{n+1}}.
\end{align*}
For $\lambda>0$, $\K_0(\Sp^n(\lambda))$ is compact. Since $\Omega^\lambda(.)$ is upper semi-continuous, there 
exists $K^*\in \K_0(\Sp^n(\lambda))$ such that $\Omega^\lambda(K^*) = C^\lambda(\alpha)$. We conjecture, that 
$K^*$ is a geodesic ball, that is, for arbitrary $p\in\Sp^n(\lambda)$, we have
\begin{align*}
	C^\lambda(\alpha) \overset{?}{=} \Omega^\lambda\left(B^n_\lambda\left(p,r\right)\right),
\end{align*}
where $r$ is determined by $\alpha= \vol_n^\lambda\left(B^n_\lambda\left(p,r\right)\right)$.

For $\lambda <0$, the problem becomes more intricate, since $\Sp^n(\lambda)$ admits unbounded closed convex 
sets with non-empty interior and finite volume. For example in hyperbolic space the ideal simplices are among them. 
Ideal simplices are simplices with vertices at infinity and they have finite hyperbolic volume. In the 
Euclidean model $(\B^n,g^h)$, such ideal simplices are just Euclidean simplices inscribed in the sphere at 
infinity $\S^{n-1}=\bd\, \B^n$. More generally, any polyhedral with vertices at infinity has finite volume. 
This is immediate by the valuation property of hyperbolic volume and the fact that any polyhedral can be 
partitioned into simplices. By monotonicity of the hyperbolic volume, we also conclude that any closed convex 
subset that is contained in a polyhedral with vertices at infinity has finite hyperbolic volume. We denote by 
$\K^{\infty}_0(\Sp^n(\lambda))$ the set of closed convex sets with non-empty interior and finite volume.
Hence, for $\lambda<0$, the space of convex bodies $\K_0(\Sp^n(\lambda))$ endowed with the symmetric 
difference metric $\theta^\lambda$ is not complete and the closure is $\K^{\infty}_0(\Sp^n(\lambda))$.

Extremizers of (\ref{eqn:liso}) could appear in $K^\infty_0(\Sp^n(\lambda))$ for $\lambda<0$, since any 
unbounded convex set in $\K^\infty_0(\Sp_n(\lambda))$ can be approximated with respect to $\theta^\lambda$ by a 
sequence of convex bodies $(K_\ell)_{\ell\in\mathbb{N}}$ in $\K_0(\Sp^n(\lambda))$ such that 
$\vol_n^\lambda(\K_\ell)=\alpha$. However, we conjecture that also in the hyperbolic setting geodesic balls 
will be extremal.

\vspace{0.2cm}
\textbf{Acknowledgement.}
The authors would like to thank the Institute for Mathematics and Applications (IMA), University of Minnesota.
It was during their stay there that part of the paper was written. We also would like to thank Monika Ludwig and Franz E.\ Schuster for valuable comments and the referee for the careful reading.

\bibliographystyle{amsplain}

\renewcommand{\bibsection}{
	\section*{\refname}
	\addcontentsline{toc}{section}{References}
}

\renewcommand{\MR}[1]{ }

\vfill
\begin{minipage}{0.4\textwidth}
	\begin{flushleft}
		\small Florian Besau\\[0.1cm]
		\tiny Department of Mathematics\\
		Case Western Reserve University\\
		Cleveland, Ohio 44106, U.S.A.\ \\[0.1cm]
		\small\texttt{florian.besau@case.edu}
	\end{flushleft}
\end{minipage}
\hfill
\begin{minipage}{0.4\textwidth}
	\begin{flushleft}
		\small Elisabeth M.\ Werner\\[0.1cm]
		\tiny Department of Mathematics\\
		Case Western Reserve University\\
		Cleveland, Ohio 44106, U.S.A.\ \\[0.1cm]
		\small \texttt{elisabeth.werner@case.edu}
	\end{flushleft}
\end{minipage}

\end{document}